\documentclass[letterpaper, 10 pt, conference]{ieeeconf}%
\usepackage{graphicx}
\usepackage{epsfig}
\usepackage{epstopdf}
\usepackage{amsmath}
\usepackage{amssymb}
\usepackage{float}
\usepackage{tikz}
\usepackage{algorithm}
\usepackage{mathrsfs}
\usepackage{subfigure}
\usepackage{booktabs}
\usepackage{multirow}
\usepackage{amsfonts}
\usepackage{wrapfig}

\IEEEoverridecommandlockouts

\newtheorem{theorem}{Theorem}
\newcounter{lemcount}

\newtheorem{proposition}[theorem]{Proposition}

\newtheorem{corollary}[theorem]{Corollary}

\newcommand{\qed}{\nobreak \ifvmode \relax \else
\ifdim\lastskip<1.5em \hskip-\lastskip
\hskip1.5em plus0em minus0.5em \fi \nobreak
\vrule height0.4em width0.5em depth0.25em\fi}

\title{{\LARGE \textbf{Event-driven Trajectory Optimization for Data Harvesting in Multi-Agent Systems}}}
\author{ \parbox{3.5 in}{\centering Yasaman Khazaeni and Christos G. Cassandras\\
         Division of Systems Engineering\\ and Center for Information and Systems Engineering\\
         Boston University, MA 02446\\
         {\tt\small yas@bu.edu,cgc@bu.edu}} \thanks{The authors' work
is supported in part by NSF under grants CNS- 1239021, ECCS-1509084, and
IIP-1430145, by AFOSR under grant FA9550-12-1-0113, by ONR under grant
N00014-09-1-1051.}}

\allowdisplaybreaks

\begin{document}
\maketitle
\thispagestyle{empty}
\pagestyle{empty}

\begin{abstract}
We propose a new event-driven method for on-line trajectory optimization to
solve the data harvesting problem: in a two-dimensional mission space, $N$
mobile agents are tasked with the collection of data generated at $M$
stationary sources and delivery to a base with the goal of minimizing expected
collection and delivery delays. We define a new performance measure that
addresses the event excitation problem in event-driven controllers and
formulate an optimal control problem. The solution of this problem provides
some insights on its structure, but it is computationally intractable,
especially in the case where the data generating processes are stochastic. We
propose an agent trajectory parameterization in terms of general function
families which can be subsequently optimized on line through the use of
Infinitesimal Perturbation Analysis (IPA). Properties of the solutions are identified, including robustness with respect to the stochastic data generation process and scalability in the size of the event set characterizing the underlying hybrid dynamical system. Explicit results are provided for the case of elliptical and Fourier series trajectories and comparisons with a state-of-the-art graph-based algorithm are given.

\end{abstract}

\section{Introduction}

Systems consisting of cooperating mobile agents have been extensively studied
and used in a broad spectrum of applications such as environmental sampling
\cite{Corke2010},\cite{Smith2011}, surveillance \cite{Tang2005}, coverage
\cite{Zhong2011},\cite{Chakrabarty2002},\cite{Cardei2005}, persistent
monitoring \cite{Alamdari2013},\cite{Cassandras2013_2}, task assignment
\cite{Panagou14}, and data harvesting and information collection
\cite{Klesh2008},\cite{Ny2008},\cite{Moazzez-Estanjini2012}.

The \emph{data harvesting} problem in particular (and its variant, the
\textquotedblleft minimum latency\textquotedblright\ problem \cite{Blum1994})
arises in many settings where wireless sensor networks (WSNs) are deployed for
purposes of monitoring the environment, road traffic, infrastructure for
transportation and for energy distribution, surveillance, and a variety of
other specialized purposes \cite{CGC2005}, \cite{Hart2006}. Although many efforts focus on the
analysis of the vast amount of data gathered, we must first ensure the
existence of robust means to collect all data in a timely fashion when the
size of the network and the level of node interference do not allow for a
fully connected wireless system. In such cases, sensors can locally gather and
buffer data, while mobile elements (e.g., vehicles, aerial drones) retrieve
the data from each part of the network. Similarly, mobile elements may
themselves be equipped with sensors and visit specific points of interest,
called \textquotedblleft targets\textquotedblright, where a direct
communication path does not exist between them and the central sink or base
node where the data must be delivered. In a delay-tolerant system, the control
scheme commonly used is to deploy mobile agents referred to \ as
\textquotedblleft data mules\textquotedblright, \textquotedblleft message
ferries\textquotedblright\ or simply \textquotedblleft
ferries\textquotedblright\ \cite{Tekdas2008},\cite{Wei2008},\cite{Zhao05}%
,\cite{Pandya2008}. The mobile agents visit the data generation nodes and
collect data which are then delivered to the base. Moreover, since agents
generally have limited buffer sizes, visits to the base are also needed once a
buffer is full; the same is true due to limited energy that requires them to
be periodically recharged. In general, the paths followed by the mobile agents
need to be optimized (in some sense to be defined) so as to ensure timely
delivery of data through sufficiently frequent visits at each data source and
the base and within the constraints of a given environment (e.g., an urban
setting).

Interestingly, there are analogs to the data harvesting problem outside the
sensor network realm. For instance, in disaster planning, evacuation and
rescue operations, pickup/delivery and transportation systems, and UAV
surveillance operations, the general theme involves a network of cooperating
mobile agents that need to frequently visit points of interest and transfer
data/goods/people to a base. Base visits may also be needed to recharge/renew
power supply/fuel. For example, the flying time span of a drone on a single
battery charge is limited, so that flight trajectories need to be optimized
and returns to base may be scheduled for recharge or loading/unloading. Thus,
we may view data harvesting in the broader context of a multi-agent system
where mobile agents must cooperatively design trajectories to visit a set of
targets and a base so as to optimize one or more performance ctiteria.

Having its root in wireless sensor networks, the data harvesting problem is
normally studied on a directed or undirected graph where minimum length tours
or sub-tours are to be found. This graph topology view of the problem utilizes
a multitude of routing and scheduling algorithms developed for wireless sensor
networks (e.g., \cite{Akkaya2005},\cite{Liu2011},\cite{Chang2014} and references therein). One
of its main advantages is the ability to accommodate environment constraints
(e.g., obstacles) by properly selecting graph edges; thus, movements inside a
building or within a road network are examples where a graph topology is
suitable. On the other hand, these methods also have several drawbacks: they
are generally combinatorially complex; they do not account for limitations in
motion dynamics which should not, for instance, allow an agent to form a
trajectory consisting of a sequence of straight lines; they become
computationally infeasible as on-line methods in the presence of stochastic
effects such as random target rewards or failing agents, since the graph
topology has to be re-evaluated as new information becomes available.

As an alternative to the graph-based approach, we view data harvesting as a
trajectory optimization problem in a two-dimensional space, where the mobile
agents are freely (or with some specified constraints) moving and
\textquotedblleft visit\textquotedblright\ targets whenever they reach their
vicinity. For example, each target can be assumed to have a finite range
within which a mobile agent can initiate wireless communication with it and
exchange data. These trajectories do not necessarily consist of straight
lines, i.e., edges in an underlying graph topology; therefore, an advantage of
this approach is that an agent trajectory can be designed to conform to
physical limitations in the agent's mobility. In addition, trajectories can be
adjusted on line when target locations are uncertain. Constraining such
trajectories to obstacles is also still possible. Such continuous topologies
have been used in \cite{Ny2008}, where the problem is viewed as a polling
system with a mobile server visiting data queues at fixed targets and
trajectories are designed to stabilize the system, keeping queue contents
(modeled as fluid queues) uniformly bounded; and in \cite{Lin2015} where
parameterized trajectories are optimized to solve a multi-agent persistent
monitoring problem.

A key benefit of a trajectory optimization view of the data harvesting problem
is the ability to parameterize the trajectories with different types of
functional representations and then optimize them over the given parameter
space. This reduces a dynamic optimization problem into a much simpler
parametric optimization one. If the parametric trajectory family is broad
enough, we can recover the true optimal trajectories; otherwise, we can
approximate them within a desired accuracy. Moreover, adopting a parametric
family of trajectories has several additional benefits. First, it allows
trajectories to be periodic, often a desirable property in practice. Second,
it allows one to restrict solutions to trajectories with desired features that
the true optimal cannot have, e.g., smoothness properties required for
physically realizable agent motion.

In this paper, we cast data harvesting as an optimal control problem. Defining
an appropriate optimization criterion is nontrivial in this problem (as we
will explain) and introducing appropriate performance metrics is the first
contribution of this work. Obtaining optimal agent trajectories ultimately
requires the solution of a two point boundary value problem \ (TPBVP).
Although a complete solution of such a TPBVP is computationally infeasible in
general, we identify structural properties of the optimal control policy which
allow us to reduce the agent-target/base interaction process to a hybrid
system with a well-defined set of events that cause discrete state
transitions. The second contribution is to formulate and solve an optimal
parametric agent trajectory problem. In particular, similar to the idea
introduced in \cite{Lin2015}, we represent an agent trajectory in terms of
general function families characterized by a set of parameters that we seek to
optimize, given an objective function. We consider elliptical trajectories as
well as the much richer set of Fourier series trajectory representations. We
then show that we can make use of Infinitesimal Perturbation Analysis (IPA)
for hybrid systems \cite{Cassandras2010} to determine gradients of the
objective function with respect to these parameters and subsequently obtain
(at least locally) optimal trajectories. This approach also allows us to
exploit $(i)$ robustness properties of IPA to allow stochastic data generation
processes, $(ii)$ the \emph{event-driven} nature of the IPA gradient
estimation process which is \emph{scalable} in the event set of the underlying
hybrid dynamic system, and $(iii)$ the \emph{on-line} computation which
implies that trajectories adjust as operating conditions change (e.g., new
targets); in contrast, the solution of a TPBVP is computationally challenging
even for strictly \emph{off line} methods. Finally, we provide comparisons of
our approach to algorithms based on a graph topology of the mission space.
These comparisons show that while the latter generate a spatial partitioning
of the target set among agents, our approach results in a temporal
partitioning which adds robustness with respect to agent failures or other
environmental changes. The graph-based approaches are mostly offline and normally assume the agents would be able to travel straight lines and meet targets in exact locations. On the other hand the trajectory optimization approach, allows us to accommodate limitations in agent mobility and to adjust trajectories on line.

In Section \ref{Formulation} we formulate the data harvesting problem using a
queueing model and present the underlying hybrid system. In Section
\ref{OptimalControl} we provide a Hamiltonian analysis leading to a TPBVP. In
Section \ref{OptimalControl} we formulate the alternative problem of
determining optimal trajectories based on general function representations and
provide solutions through a gradient-based algorithm using IPA for two
particular function families. Section \ref{Numerical} presents numerical
results and comparisons with state of the art data harvesting algorithms and
Section \ref{Conclusion} contains conclusions.

\section{Problem Formulation}

\label{Formulation} We consider a data harvesting problem where $N$ mobile
agents collect data from $M$ stationary targets in a two-dimensional mission
space $S$. Each agent may visit one or more of the $M$ targets, collect data
from them, and deliver them to a base. It then continues visiting targets,
possibly the same as before or new ones, and repeats this process. The objective of the agent team is to minimize data collection and delivery delays over all targets within a fixed time interval $T$. This minimization problem is formalized in the sequel.

The data harvesting problem described above can be viewed as a polling system
where mobile agents are serving the targets by collecting data and delivering
it to the base. As seen in Fig. \ref{queueschematic}, there are three sets of
queues. The first set includes the data contents $X_{i}(t)\in\mathbb{R}^{+}$
at each target $i=1,...,M$ where we use $\sigma_{i}(t)$ as the instantaneous
inflow rate. In general, we treat $\{\sigma_{i}(t)\}$ as a random process
assumed only to be piecewise continuous; we will treat it as a deterministic
constant only for the Hamiltonian analysis in the next section. Thus, at time
$t$, $X_{i}(t)$ is a random variable resulting from the random process
$\{\sigma_{i}(t)\}$.

The second set of queues consists of data contents $Z_{ij}(t)\in\mathbb{R}%
^{+}$ onboard agent $j$ collected from targets $i=1,...,M$. The last set
consists of queues $Y_{i}(t)\in\mathbb{R}^{+}$ containing data at the base,
one queue for each target, delivered by some agent $j$. Note that
$\{Z_{ij}(t)\}$ and $\{Y_{i}(t)\}$ are also random processes.

In Fig. \ref{queueschematic} collection and delivery switches are shown by
$p_{ij}$ and $p_{\!_{Bj}}$ (formally defined in the sequel). These switches
are \textquotedblleft on" when agent $j$ is connected to target $i$ or the
base respectively. All queues are modeled as flow systems whose dynamics are
given next (however, as we will see, the agent trajectory optimization is
driven by events observed in the underlying system where queues contain
discrete data packets so that this modeling device has minimal effect on our
analysis). \begin{figure}[ptb]
\begin{center}
\begin{tikzpicture}[scale=0.50]
\draw[thin,gray](0,0)--(10,0);
\draw[very thick,blue] (0.5,1)--(0.5,0)--(1,0)--(1,1);
\draw [fill,blue, nearly transparent]  (0.5,0.8)--(0.5,0)--(1,0)--(1,0.8);
\draw [black,thick,->] (0.75,2.0)--(0.75,1.0);
\draw (0.1,1.2) node {$X_1$};
\draw (4.75,0.6) node {$\dots$};
\draw (4.75,1.0) node {$X_i$};
\draw[very thick,blue] (8.5,1)--(8.5,0)--(9,0)--(9,1);
\draw [fill,blue, nearly transparent]  (8.5,0.4)--(8.5,0)--(9,0)--(9,0.4);
\draw [black,thick,->] (8.75,2.0)--(8.75,1.0);
\draw (9.5,1.2) node {$X_M$};
\draw[blue] (-0.1,-0.5) node {$p_{11}$};
\draw[blue] (10.0,-0.5) node {$p_{MN}$};
\draw[very thick,red] (0.5,-0.5)--(1,-0.2);
\draw[very thick,red,fill] (1,-0.2) circle(0.05);
\draw[very thick,red] (4.7,-0.3) node {$\ldots$};
\draw[very thick,red] (8.5,-0.5)--(9,-0.2);
\draw[very thick,red,fill] (9,-0.2) circle(0.05);
\draw[thin,gray](0,-1.1)--(10,-1.1);
\draw[thin,gray](0,-0.5)--(10,-0.5);
\draw[thin,gray](4.5,-1.1)--(4.5,-2.5);
\draw[very thick,green] (5.5,-1.5)--(4.5,-1.5)--(4.5,-2)--(5.5,-2);
\draw [fill,blue, nearly transparent]  (5.2,-1.5)--(4.5,-1.5)--(4.5,-2)--(5.2,-2);
\draw [black,thick,->] (6.5,-1.2)--(6.5,-1.75)--(5.5,-1.75);
\draw (7,-1.75) node {$Z_{ij}$};
\draw[thin,gray](0,-2.5)--(10,-2.5);
\draw[blue] (-0.1,-2.7) node {$p_{B1}$};
\draw[blue] (10.0,-2.7) node {$p_{BN}$};
\draw[very thick,red] (0.5,-3)--(1,-2.7);
\draw[very thick,red,fill] (1,-2.7) circle(0.05);
\draw[very thick,red] (4.7,-2.8) node {$\ldots$};
\draw[very thick,red] (8.5,-3.0)--(9,-2.7);
\draw[very thick,red,fill] (9,-2.7) circle(0.05);
\draw[thin,gray](0,-3)--(10,-3);
\draw[very thick,red] (0.5,-4.5)--(0.5,-3.5)--(1,-3.5)--(1,-4.5);
\draw [fill,blue, nearly transparent]  (0.5,-4.2)--(0.5,-3.5)--(1,-3.5)--(1,-4.2);
\draw [black,thick,->] (1.5,-3.8)--(1.5,-5)--(0.7,-5)--(0.7,-4.5);
\draw (0.25,-4.8) node {$Y_{1}$};
\draw (4.75,-4.2) node {$\ldots$};
\draw[very thick,red] (8.5,-4.5)--(8.5,-3.5)--(9,-3.5)--(9,-4.5);
\draw [fill,blue, nearly transparent]  (8.5,-4.2)--(8.5,-3.5)--(9,-3.5)--(9,-4.2);
\draw [black,thick,->] (8,-3.8)--(8,-5)--(8.8,-5)--(8.8,-4.5);
\draw (9.5,-4.8) node {$Y_{M}$};
\draw[thin,gray](0,-3.5)--(10,-3.5);
\draw (4.75,-4.75) node {$Y_{i}$};
\end{tikzpicture}
\end{center}
\caption{Data harvesting queueing model for $M$ targets and $N$ agents}%
\label{queueschematic}%
\end{figure}
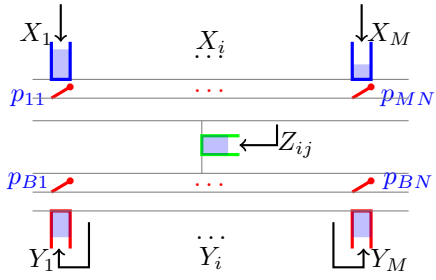

Let $s_{j}(t)=[s_{j}^{x}(t),s_{j}^{y}(t)]\in S$ be the position of agent $j$
at time $t$, Then, the state of the system can be defined as
\begin{align}
\label{state}
\mathbf{X}  &  (t)=[X_{1}(t),\dots,X_{M}(t),Y_{1}(t),\dots,Y_{M}(t),\\\nonumber
&  Z_{11}(t),\dots,Z_{MN}(t),s_{1}^{x}(t),s_{1}^{y}(t),\dots,s_{N}%
^{x}(t),s_{N}^{y}(t)] %
\end{align}
The position of the agent follows single integrator dynamics at all times:
\begin{equation}
\dot{s}_{j}^{x}(t)=u_{j}(t)\cos\theta_{j}(t),\quad\ \ \dot{s}_{j}^{y}%
(t)=u_{j}(t)\sin\theta_{j}(t) \label{agentdynamics}%
\end{equation}%
\[
s_{j}^{x}(0)=X_{B}\qquad s_{j}^{y}(0)=Y_{B},\quad\forall j
\]
where $u_{j}(t)$ is the scalar speed of the agent (normalized so that $0\leq
u_{j}(t)\leq1$), $0\leq\theta_{j}(t)<2\pi$ is the angle relative to the
positive direction and $[X_{B},Y_{B}]$ is the location of the base. Thus, we
assume that the agent controls its orientation and speed. The agent
states $\{s_{j}(t)\}$, $j=1,\dots,N$, are also random processes since the
controls are generally dependent on the random queue states. Thus, we ensure
that all random processes are defined on a common probability space.

An agent is represented as a particle, so that we will omit the need for any
collision avoidance control. The agent dynamics above could be more
complicated without affecting the essence of our analysis, but we will limit
ourselves here to (\ref{agentdynamics}).
%For every two points ${w}\in S$ and ${v}\in S$ we associate a function $p(w,v)$ where $p(w,v) = 1$ if $w = v$, and $p(w,v)$ is monotonically non-increasing in the value of Euclidean distance $D(w,v)=\|w-v\|$. A simple linear decay model for $p$ is as following:
%\begin{equation}
%p(w, v) = \left\{
%\begin{array}{l l}
%1-\frac{D(w,v)}{r(w)}\quad & \mbox{if } D(w,v)\le r(w)\\
%0 & \mbox{if } D(w,v) > r(w)
%\end{array}
%\right.
%\label{pdefinition}
%\end{equation}
%where $r(w)$ is a finite sensing range for $w$ meaning, $p(w, v) = 0$ when $D(w,v)> r(w)$.

We consider a set of data sources as points $w_{i}\in S,$ $i=1,\dots,M,$ with
associated ranges $r_{ij}$, so that agent $j$ can collect data from $w_{i}$
only if the Euclidean distance $d_{ij}(t)=\Vert w_{i}-s_{j}(t)\Vert$ satisfies
$d_{ij}(t)\leq r_{ij}$. Similarly, the base is at $w_{\!_{B}}=[X_{B},Y_{B}]\in
S$ which receives all data collected by the agents and an agent can only
deliver data to the base if the Euclidean distance $d_{\!_{Bj}}(t)=\Vert
w_{\!_{Bj}}-s_{j}(t)\Vert$ satisfies $d_{\!_{Bj}}(t)\leq r_{Bj}$. Using
$p:S\times S\rightarrow\lbrack0,1]$, we define a function $p_{ij}(t)$
representing the collection switches in Fig. \ref{queueschematic} as:
\begin{equation}
p_{ij}(t)=p(w_{i},s_{j}(t)) \label{Pij}%
\end{equation}
$p_{ij}(t)$ is viewed as the normalized data collection rate from target $i$
when the agent is at $s_{j}(t)$ and we assume that: $(\mathbf{A1})$ it is
monotonically non-increasing in the value of $d_{ij}(t)=\Vert w_{i}%
-s_{j}(t)\Vert$, and $(\mathbf{A2})$ it satisfies $p_{ij}(t)=0$ if
$d_{ij}(t)>r_{ij}$. Thus, $p_{ij}(t)$ can model communication power
constraints which depend on the distance between a data source and an agent
equipped with a receiver (similar to the model used in \cite{Ny2008}) or
sensing range constraints if an agent collects data using on-board sensors.
For simplicity, we will also assume that: $(\mathbf{A3})$ $p_{ij}(t)$ is
continuous in $d_{ij}(t)$. Similarly, we define:
\begin{equation}
p_{\!_{Bj}}(t)=p(w_{\!_{B}},s_{j}(t)) \label{PB}%
\end{equation}
The maximum rate of data collection from target $i$ by agent $j$ is $\mu_{ij}%
$, so that the instantaneous rate is $\mu_{ij}p_{ij}(t)$ if $j$ is connected
to $i$. We will assume that: $(\mathbf{A4})$ only one agent at a time is
connected to a target $i$ even if there are other agents $l$ with
$p_{il}(t)>0$; this is not the only possible model, but we adopt it based on
the premise that simultaneous downloading of packets from a common source
creates problems of proper data reconstruction at the base.

We can now define the dynamics of the queue-related components of the state
vector in \eqref{state}. The dynamics of $X_{i}(t)$, assuming that agent $j$
is connected to it, are
\begin{equation}
\resizebox{0.9 \columnwidth}{!}{$
\dot{X}_{i}(t)=\left\{
\begin{array}
[c]{ll}%
0\qquad\quad\mbox{if }X_{i}(t)=0\mbox{ and }\sigma_{i}(t)\leq\mu_{ij}%
p_{ij}(t) & \\
\sigma_{i}(t)-\mu_{ij}p_{ij}(t)\quad\qquad\qquad\qquad\mbox{otherwise} &
\end{array}
\right.  \label{Xdot}
$}
\end{equation}
Obviously, $\dot{X}_{i}(t)=\sigma_{i}(t)$ if $p_{ij}(t)=0$, $j=1,\dots,N$.

In order to express the dynamics of $Z_{ij}(t)$, let
\begin{equation}
\resizebox{0.99 \columnwidth}{!}{$
\tilde{\mu}_{ij}(t)=\left\{
\begin{array}
[c]{ll}%
\min\Big(\frac{\sigma_{i}(t)}{p_{ij}(t)},\mu_{ij}\Big)\quad & \mbox{if }X_{i}%
(t)=0\text{ and }p_{ij}(t)>0\\
\mu_{ij}\quad & \mbox{otherwise}
\end{array}
\right.  \label{mutilde}
$}
\end{equation}
This gives us the dynamics:
\begin{equation}
\resizebox{0.91 \columnwidth}{!}{$
\dot{Z}_{ij}(t)=\left\{
\begin{array}
[c]{ll}%
0\qquad\mbox{if }Z_{ij}(t)=0\mbox{ and }\tilde{\mu}_{ij}(t)p_{ij}%
(t)-\beta_{ij}p_{\!_{Bj}}(t)\leq0 & \\
\tilde{\mu}_{ij}(t)p_{ij}(t)-\beta_{ij}p_{\!_{Bj}}(t)\qquad\qquad\qquad
\qquad\mbox{otherwise} &
\end{array}
\right.  \label{Zdot}
$}
\end{equation}
where $\beta_{ij}$ is the maximum rate of data from target $i$ delivered to
$B$ by agent $j$. For simplicity, we assume that: $(\mathbf{A5})$ $\Vert
w_{i}-w_{B}\Vert>r_{ij}+r_{Bj}$ for all $i=1,\dots,M$ and $j=1,\dots,N$, i.e.,
the agent cannot collect and deliver data at the same time. Therefore, in
(\ref{Zdot}) it is always the case that for all $i$ and $j$, $p_{ij}%
(t)p_{Bj}(t)=0$. Finally, the dynamics of $Y_{i}(t)$ depend on $Z_{ij}(t)$, the content of the
on-board queue of each agent $j$ from target $i$ as long as $p_{\!_{Bj}}(t)>0$. We
define $\beta_{i}(t)=\sum_{j=1}^{N}\beta_{ij}p_{\!_{Bj}}(t)\mathbf{1}[Z_{ij}(t)>0]$
as the total instantaneous delivery rate for target $i$ data, so that the
dynamics of $Y_{i}(t)$ are:
\begin{equation}
\dot{Y}_{i}(t)=\beta_{i}(t) \label{Ydot}%
\end{equation}

\noindent\textbf{Hybrid System model:} Taking into account the state vector in
\eqref{state} and the dynamics in \eqref{agentdynamics}, \eqref{Xdot},
\eqref{Zdot} and \eqref{Ydot}, the data harvesting process is a stochastic
hybrid system. Discrete modes of the system are defined by intervals over
which $(i)$ agents are visiting a target, $(ii)$ agents are visiting the base,
and $(iii)$ agents are moving when not connected to any target or base. The
events that trigger mode transitions are defined in Table \ref{eventlist} (the
superscript $0$ denotes events causing a variable to reach a value of zero
from above and the superscript $+$ denotes events causing a variable to become
strictly positive from a zero value). We also use the following definitions:
\begin{equation}
\resizebox{0.9 \columnwidth}{!}{$
d_{ij}^{+}(t)=\max(0,d_{ij}(t)-r_{ij}),\text{\ }d_{\!_{Bj}}^{+}(t)=\max
(0,d_{\!_{Bj}}(t)-r_{\!_{Bj}})\label{Dplus}%
$}
\end{equation}
The variables above are zero if agent $j$ is within range of target $i$ or the base
respectively.

\begin{table}[tbh]
\caption{Hybrid System Events}%
\label{eventlist}
\renewcommand{\arraystretch}{1.2} \centering
\begin{tabular}
[c]{|c|l|}\hline
Event Name & Description\\\hline
1. $\xi_{i}^{0}$ & $X_{i}(t)$ hits 0, for $i=1,\dots,M$\\\hline
2. $\xi_{i}^{+}$ & $X_{i}(t)$ leaves 0, for $i=1,\ldots,M$.\\\hline
3. $\zeta_{ij}^{0}$ & $Z_{ij}(t)$ hits 0, for $i=1,\ldots,M$, $j=1,\ldots
,N$\\\hline
4. $\delta_{ij}^{+}$ & $d_{ij}^{+}(t)$ leaves 0, for $i=1,\ldots,M$,
$j=1,\ldots,N$\\\hline
5. $\delta_{ij}^{0}$ & $d_{ij}^{+}(t)$ hits 0, for $i=1,\ldots,M$,
$j=1,\ldots,N$\\\hline
6. $\Delta_{j}^{+}$ & $d_{\!_{Bj}}^{+}(t)$ leaves 0, for $j=1,\ldots
,N$\\\hline
7. $\Delta_{j}^{0}$ & $d_{\!_{Bj}}^{+}(t)$ hits 0, for $j=1,\ldots,N$\\\hline
\end{tabular}
\end{table}\begin{figure}[ptb]
\centering
%Requires \usepackage{graphicx}
\includegraphics[width=3in]{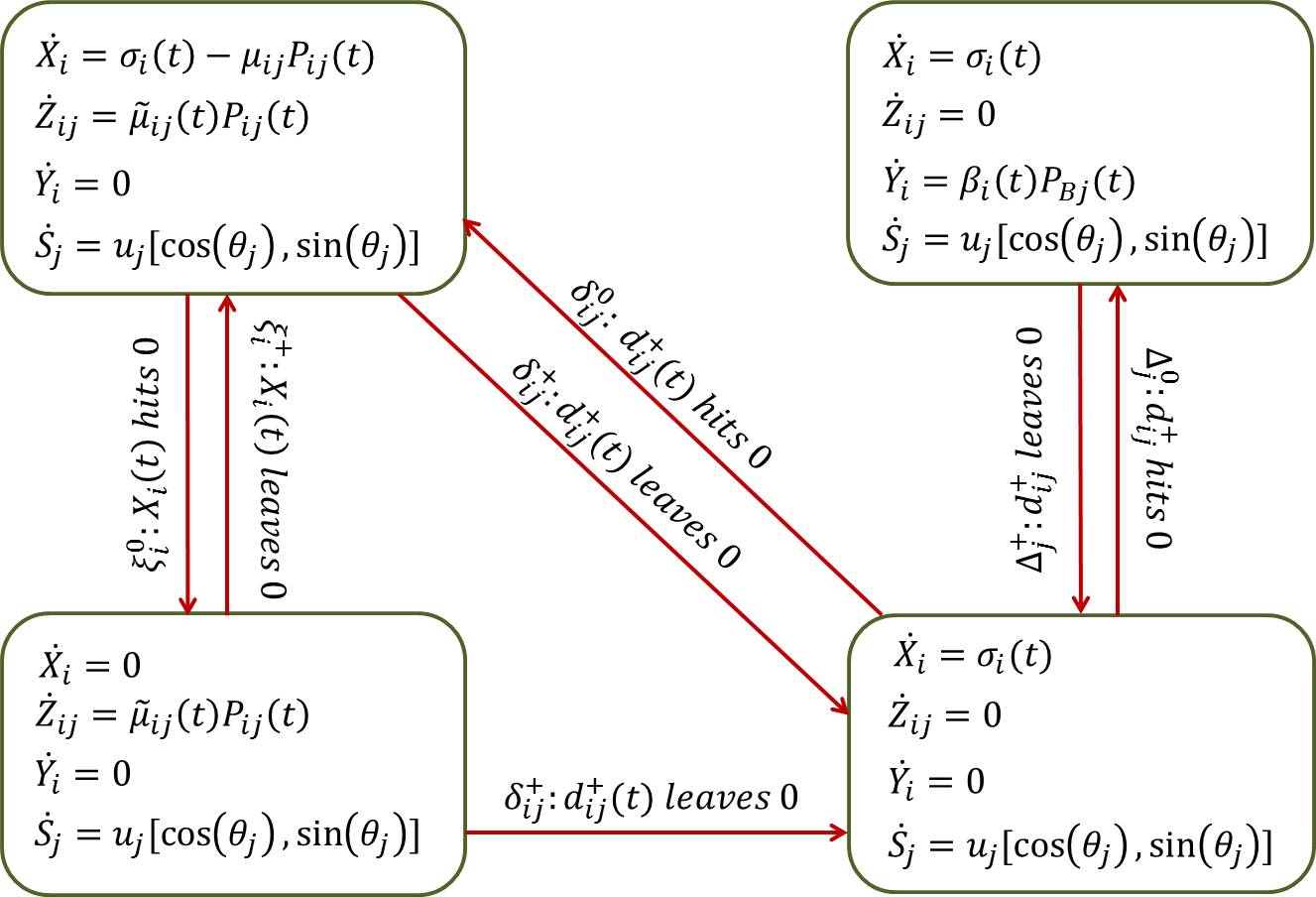}\newline\caption{One target $i$ and one
agent $j$ hybrid automaton}%
\label{Automaton}%
\end{figure}Observe that each of the events in Table \ref{eventlist} causes a
change in at least one of the state variables in \eqref{Xdot}, \eqref{Zdot},
\eqref{Ydot}. For example, $\xi_{i}^{0}$ (i.e., the queue at target $i$ is
emptied) causes a switch in (\ref{Xdot}) from $\dot{X}_{i}(t)=\sigma
_{i}(t)-\mu_{ij}p_{ij}(t)$ to $\dot{X}_{i}(t)=0$. Also note that we have
omitted an event $\zeta_{ij}^{+}$ for $Z_{ij}(t)$ becoming strictly positive
since this event is immediately induced by $\delta_{ij}^{0}$ when agent $j$
comes within range of target $i$ and starts collecting data causing
$Z_{ij}(t)>0$ if $Z_{ij}(t)=0$ and $X_{i}(t)>0$. Finally, note that all events
are directly observable during the execution of any agent trajectory and they
do not depend on our flow-based queueing model. For example, if $X_{i}(t)$
becomes zero, this defines event $\xi_{i}^{0}$ regardless of whether the
corresponding queue is based on a flow or on discrete data packets; this
observation is very useful in the sequel. A high-level hybrid automaton is
presented in Fig. \ref{Automaton} for a single target
$i$ and one agent $j$ system. This automaton becomes much more complicated once more
targets and agents are included.

\subsection{Performance Measures\label{performance}}

Our objective is to maintain minimal data content at all target queues while
also maximizing the contents of the delivered data at the base queues. Thus,
we define $J_{1}(t)$ to be the weighted sum of expected target queue content
(recalling that $\{\sigma_{i}(t)\}$ are random processes):
\begin{equation}
J_{1}(t)=E\big[\sum\limits_{i=1}^{M}\alpha_{i}X_{i}(t)\big] \label{J1}%
\end{equation}
where the weight $\alpha_{i}$ represents the relative importance factor of
target $i$. Similarly, we define a weighted sum of expected base queue
content:
\begin{equation}
J_{2}(t)=E\big[\sum\limits_{i=1}^{M}\alpha_{i}Y_{i}(t)\big] \label{J2}%
\end{equation}
Therefore, a tentative optimization objective is the convex combination of
(\ref{J1}) and (\ref{J2}) leading to the minimization problem:
\begin{equation}
\min\limits_{\mathbf{u(t),\boldsymbol{\theta}(t)}}J(T)=\frac{1}{T}\int_{0}%
^{T}\Big(qJ_{1}(t)-(1-q)J_{2}(t)\Big)dt \label{OptimJ1J2}%
\end{equation}
where $\mathbf{u}$ and $\boldsymbol{\theta}$ are the vectors formed by the agent speed and headings and $q\in[0,1]$ is a weight capturing the relative importance of
collected data as opposed to delivered data.

This performance measure captures the collection and delivery of data which
are processes taking place while an agent is connected to any of the targets
or the base. However, it lacks any information regarding the interaction of an
agent with the environment when this agent is not connected to any target or
base and is due to the fact that the environment has only a finite number of
points of interest (targets). This motivates two new performance measures we
introduce next.

\noindent\textbf{Agent Utilization:} In accessing the targets, we must ensure
that the agents maximize their utilization, i.e., the fraction of time spent
performing a useful task by being within range of a target or the base.
Equivalently, we aim to minimize the non-productive idling time of each agent
during which it is not visiting any target or the base. Using (\ref{Dplus}),
agent $j$ is idling when $d_{ij}^{+}(t)>0$ for all $i$ and $d_{\!_{Bj}}%
^{+}(t)>0$. We define the idling function $I_{j}(t)$ as follows:
\begin{equation}
I_{j}(t)=\log\Bigg(1+d_{\!_{Bj}}^{+}(t)\prod_{i=1}^{M}d_{ij}^{+}%
(t)\Bigg) \label{Ij}%
\end{equation}
This function has the following properties. First, $I_{j}(t)=0$ if and only if
the product term inside the bracket is zero, i.e., agent $j$ is visiting a
target or the base; otherwise, $I_{j}(t)>0$. Second, $I_{j}(t)$ is
monotonically nondecreasing in the number of targets $M$. The logarithmic
function is selected to prevent the value of $I_{j}(t)$ from dominating
those of $J_{1}(\cdot)$ and $J_{2}(\cdot)$ when included in a single objective
function. Thus, we define:
\begin{equation}
J_{3}(t)=E\big[\sum\limits_{j=1}^{N}I_{j}(t)\big] \label{J3}%
\end{equation}
Note that $I_{j}(t)$ is also a random variable since it is a function of the
agent states $s_{j}(t)$, $j=1,\dots,N$.

\noindent\textbf{Event Excitation:} As mentioned in the Introduction, our goal
is to develop an event-driven approach for on-line trajectory optimization. In
other words, we seek a controller whose actions are based on events observed
during the operation of the hybrid system described earlier. Clearly, the
premise of this approach is that the events involved are observable so as to
\textquotedblleft excite\textquotedblright\ the underlying event-driven
controller. However, it is not always obvious that these events actually take
place under every feasible control, in which case the controller may be
useless. This is illustrated in Fig. \ref{twotrajectories} where two different
trajectories are shown for the agent. The blue and red trajectories pass
through none of the targets. Consequently, there is an infinite number of
trajectories for which the value of the objective function in
\eqref{OptimJ1J2} is given by
\begin{equation}
J(T)=\frac{q}{T}\int_{0}^{T}t\sigma_{i}(t)dt\
\end{equation}
which is simply the total amount of data generated at all targets through
(\ref{Xdot}). This cannot be affected by any event-driven control action,
since none of the events in Table \ref{eventlist} is excited. Clearly, the
same is true for $J_{3}(t)$ in (\ref{J3}). \begin{figure}[ptb]
\centering
\includegraphics[width=2.2in]{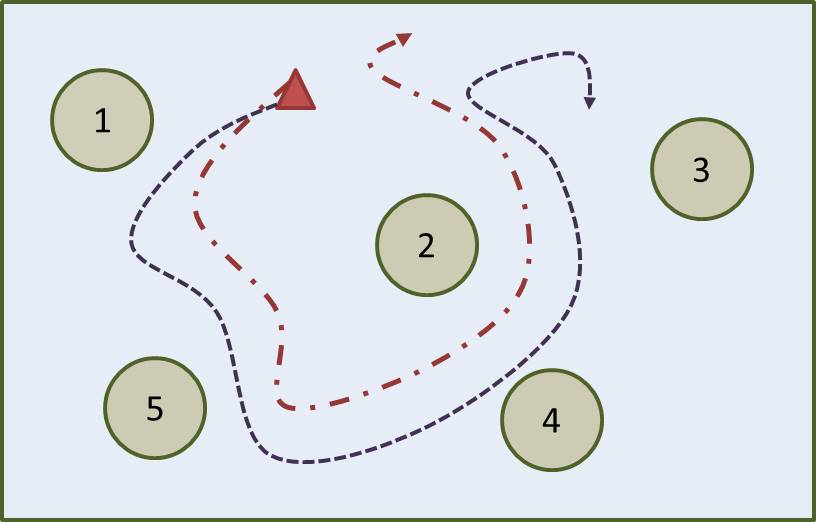}\caption{Two trajectories with
same objective function value}%
\label{twotrajectories}%
\end{figure}

To address this issue, our goal is to \textquotedblleft
spread\textquotedblright\ each target cost (accumulated data) $X_{i}(t)$ over
all $w\in S$. This will create a potential field throughout the mission space.
Following \cite{Khazaeni2016}, we begin by determining the convex hull
produced by the targets, since the trajectories need not go outside this
polygon. Let $\mathcal{T}=\{w_{1},w_{2},\cdots,w_{M}\}$ be the set of all
target points. Then, their convex hull is
\begin{equation}
\mathcal{C}=\bigg\{\sum_{i=1}^{M}\beta_{i}w_{i}|\sum_{i}\beta_{i}=1,\forall
i,~\beta_{i}\geq0\bigg\}
\end{equation}
Given that $\mathcal{C}\subset S$, we seek a function $R(w,t)$ that satisfies
the following property for some constants $c_{i}>0$:
\begin{equation}
\int_{\mathcal{C}}R(w,t)dw=\sum_{i=1}^{M}c_{i}X_{i}(t)\label{Rproperty}%
\end{equation}
Thus, $R(w,t)$ can be viewed as a time-varying density function defined for
all points $w\in\mathcal{C}$ which generates a total cost equivalent to a
weighted sum of the target data $X_{i}(t)$, $i=1,\ldots,M$. Letting $d_{i}%
^{+}(w)=\max(\Vert w-w_{i}\Vert,r_{i})$, where $r_{i}=\min_{j}{r_{ij}}$, we
then define:
\begin{equation}
R(w,t)=\sum_{i=1}^{M}\frac{\alpha_{i}X_{i}(t)}{d_{i}^{+}(w)}\label{J2R}%
\end{equation}
Intuitively, a target's cost (numerator above) is spread over all $w\in S$ so
as to obtain the \textquotedblleft total weighted cost
density\textquotedblright\ at $w$. Note that $d_{i}^{+}(w)$ is defined to
ensure that the target cost remains positive and fixed for all points $w\in
C(w_{i})$. In order to illustrate this construction, Fig. \ref{Qa} shows a
sample mission space with 9 target locations and Fig \ref{Qb} shows the value
of $R(w,t)$ at a specific time $t$.

\begin{proposition}
There exist $c_{i}>0$, $i=1,\ldots,M$, such that:
\begin{equation}
\int_{\mathcal{C}}R(w,t)dw=\sum_{i=1}^{M}c_{i}x_{i}(t) \label{corequation}%
\end{equation}
\label{Rproof}
\end{proposition}
\begin{proof}
See Appendix \ref{appendixproofs}.
\end{proof}

Using the same idea for the base, we define
\begin{equation}
R_{\!_{Bj}}(w,t)=\frac{\sum_{i=1}^{M}\alpha_{i}Z_{ij}}{d_{\!_{B}}^{+}(w)}%
\end{equation}
where $d_{\!_{B}}^{+}(w)=\max(\Vert w_{\!_{B}}-w\Vert,r_{\!_{B}})$ is a
constant and $r_{\!_{B}}=\min_{j}{r_{\!_{Bj}}}$.

Proposition \ref{Rproof} asserts that the total cost due to data accumulated at targets
may indeed be spread over all points in the mission space allowing an agent to
\textquotedblleft interact\textquotedblright\ with these points through the
resulting potential field. In order to capture this interaction (see also
\cite{Khazaeni2016}), we define the travel cost for an agent $j$ to reach
point $w$ as the quadratic of the distance between them $\Vert s_{j}%
(t)-w\Vert^{2}$ and the total travel cost as
\begin{equation}
P(w,\mathbf{s}(t))=\sum_{j=1}^{N}\Vert s_{j}(t)-w\Vert^{2} \label{Pfunction2}%
\end{equation}
Using these definitions we can now introduce a new performance metric:
\begin{equation}
\resizebox{0.89 \columnwidth}{!}{$
J_{4}(t)=E\big[\sum_{j=1}^{N}\int_{S}\bigg(R(w,t)+R_{\!_{Bj}}(w,t)\bigg)P_{j}%
(w,t)dw\big]$} \label{J4}%
\end{equation}
\begin{figure}[ptb]
\centering
\begin{subfigure}[Mission Space with dots as target locations]{
\includegraphics[width=1.55in,height=1.55in]{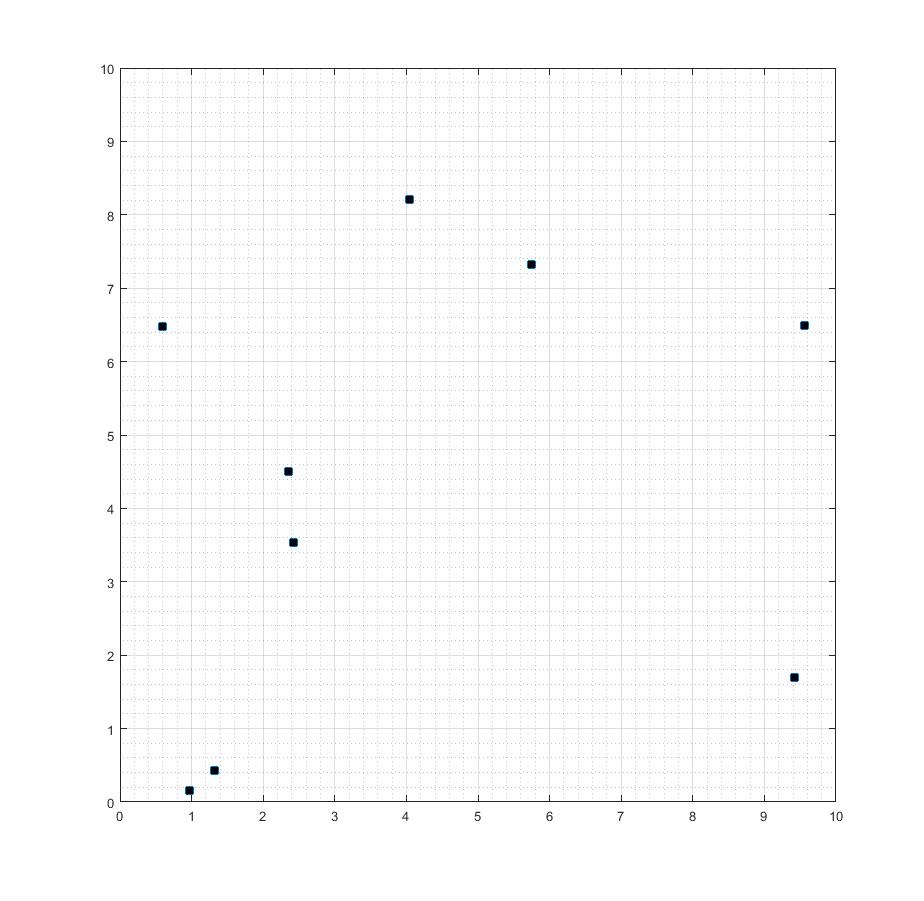}\label{Qa}}\end{subfigure}
\hspace{-2mm} \begin{subfigure}[$R$ Function at a sample time $t$]{
\includegraphics[width=1.65in,height=1.55in]{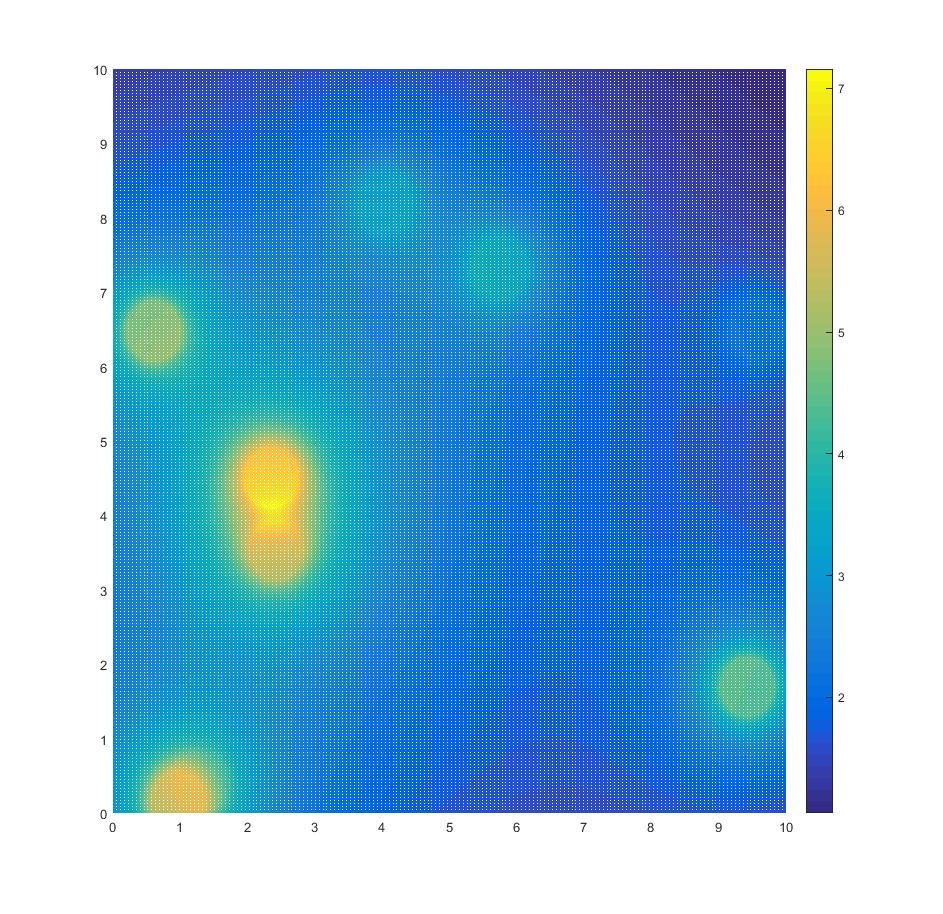}\label{Qb}}
\end{subfigure}
\caption{$R$ function illustration}%
\end{figure}
\noindent\textbf{Terminal Cost:} Since we address the data harvesting problem
over a finite interval $T$, we define a terminal cost at $T$ capturing the
expected value of the amount of data left on board the agents:
\begin{equation}
J_{f}(T)=\frac{1}{T}E\big[\sum\limits_{i=1}^{M}\sum\limits_{j=1}^{N}\alpha
_{i}Z_{ij}(T)\big] \label{Jf}%
\end{equation}
Clearly, the effect of this term vanishes as $T\rightarrow\infty$ as long as
all $E[Z_{ij}(T)]$ remains bounded. Moreover, if we constrain trajectories to
be periodic, this terminal cost may be omitted. Finally, for simplicity, we
will assume that $\alpha_{i}=1$ for all $i$.

\subsection{Optimization Problem}

We can now formulate a stochastic optimization problem $\mathbf{P1}$ where the
control variables are the agent speeds and headings denoted by the vectors
$\mathbf{u}(t)=[u_{1}(t),\dots,u_{N}(t)]$ and $\boldsymbol{\theta}%
(t)=[\theta_{1}(t),\dots,\theta_{N}(t)]$ respectively (omitting their
dependence on the full system state at $t$). Combining the components in \eqref{J3}, \eqref{J4} and \eqref{Jf} we obtain:
\begin{align}
\min\limits_{\mathbf{u(t),\boldsymbol{\theta}(t)}}J(T)= &  \frac{1}{T}\int
_{0}^{T}\Big(\frac{q}{M_{X}}J_{1}(t)-\frac{(1-q)}{M_{Y}}J_{2}(t)\\\nonumber
&  +\frac{1}{M_{I}}J_{3}(t)+\frac{1}{M_{R}}J_{4}(t)\Big)dt+\frac{1}{M_{Z}%
}J_{f}(T)\label{GenOptim}%
\end{align}
where we introduce the normalizing factors $M_{X}$, $M_{Y}$, $M_{I}$, $M_{R}$
and $M_{Z}$. This normalization ensures that all different components of
$J(T)$ are in the same range so that none of them may dominate any other. We
use an upper bound for the value of each component as follows, where we assume that $\sigma_i(0)>0$ w.p. 1:
\begin{align}
M_{X} &  =M_{Y}=M_{Z}=T\sum_{i}\sigma_{i}(0)\\
M_{I} &  =\log\Big(1+\sqrt{L_{1}^{2}+L_{2}^{2}}^{M+1}\Big)\\
M_{R} &  =\frac{TL_{1}L_{2}(L_{1}^{2}+L_{2}^{2})}{r}\sum_{i}\sigma
_{i}(0),r=\frac{\sum_{i}r_{i}}{M}
\end{align}
where $L_{1}$ and $L_{2}$ define the size of the rectangular mission space
(the normalization factors can easily be adapted to different mission space
shapes). Observe that an unattainable lower bound of the total objective
function is $-(1-q)$ which occurs if $J_{1}=J_{3}=J_{4}=0$ and $J_{2}$ is at
its maximum of 1. If $q=0$, then the lower bound is at its minimum of -1.
\section{Optimization Methodology}
\label{OptimalControl} In this section, we consider problem $\mathbf{P1}$ in a
setting where all data arrival processes are deterministic, so that all
expectations in \eqref{J1}-\eqref{Jf} degenerate to their arguments. We
proceed with a standard Hamiltonian analysis leading to a Two Point Boundary
Value Problem (TPBVP) \cite{bryson1975applied} where the states and costates
are known at $t=0$ and $t=T$ respectively. We define the costate vector
associated to \eqref{state}:%
\begin{align}
\nonumber
\boldsymbol{\lambda}  &  (t)=[\lambda_{1}(t),\dots,\lambda_{M}(t),\gamma
_{1}(t),\dots,\gamma_{M}(t),\\
&  \phi_{11}(t),\dots,\phi_{MN}(t),\eta_{1}^{x}(t),\eta_{1}^{y}(t),\dots
,\eta_{N}^{x}(t),\eta_{N}^{y}(t)]
\end{align}
The Hamiltonian is
\begin{align}
\nonumber
&  H(\mathbf{X},{\boldsymbol{\lambda}},\mathbf{u},{\boldsymbol{\theta}}%
)=\frac{1}{T}\Big[qJ_{1}(t)-(1-q)J_{2}(t)+J_{3}(t)+J_{4}(t)\Big]\\\nonumber
&  +\sum_{i}\lambda_{i}(t)\dot{X}_{i}(t)+\sum_{i}\gamma_{i}(t)\dot{Y}%
_{i}(t)+\sum_{i}\sum_{j}\phi_{ij}(t)\dot{Z}_{ij}(t)\\
&  +\sum_{j}\big(\eta_{j}^{x}(t)u_{j}(t)\cos\theta_{j}(t)+\eta_{j}^{y}%
(t)u_{j}(t)\sin\theta_{j}(t)\big)
\label{hamiltonian}%
\end{align}
where the costate equations are {\small
\[%
\resizebox{0.99 \columnwidth}{!}{$
\begin{split}
&  \dot{\lambda}_{i}(t)=-\frac{\partial H}{\partial{X_{i}}}=-\frac{1}%
{T}\big[\frac{q}{M_{X}}+\frac{1}{M_{R}}\sum\limits_{j}\int_{S}\frac{\alpha
_{i}P_{j}(w,t)}{d_{i}^{+}(w)}dw\big]~\lambda_{i}(T)=0\\
&  \dot{\gamma}_{i}(t)=-\frac{\partial H}{\partial{Y_{i}}}=\frac{1-q}{TM_{Y}%
}\quad\gamma_{i}(T)=0
\end{split}$}
\]
\[
\resizebox{0.99 \columnwidth}{!}{$
\begin{split}
&  \dot{\phi}_{ij}(t)=-\frac{\partial H}{\partial{Z_{ij}}}=-\frac{1}{M_{R}%
}\int_{S}\frac{\alpha_{i}P_{j}(w,t)}{d_{\!_{B}}^{+}(w)}dw\quad\phi
_{ij}(T)=\frac{\partial J_{f}}{\partial Z_{ij}}\Big|_{T}%
\end{split}$}
\]%
\[%
\resizebox{0.99 \columnwidth}{!}{$
\begin{split}
\dot{\eta}_{j}^{x}(t)=  &  -\frac{\partial H}{\partial s_{j}^{x}}\\
=  &  -\Bigg[\frac{1}{TM_{I}}\frac{\partial I_{j}(t)}{\partial s_{j}^{x}%
}+\frac{1}{TM_{R}}\sum_{j}\int_{S}\big(R(w,t)+R_{\!{Bj}}(w,t)\big)\frac
{\partial P_{j}(w,t)}{\partial s_{j}^{x}}dw\\
&  +\sum_{i}\frac{\partial}{\partial s_{j}^{x}}\lambda_{i}(t)\dot{X}%
_{i}(t)+\sum_{i}\frac{\partial}{\partial s_{j}^{x}}\gamma_{i}(t)\dot{Y}%
_{i}(t)+\sum_{i}\frac{\partial}{\partial s_{j}^{x}}\phi_{ij}(t)\dot{Z}%
_{ij}(t)\Bigg]
\end{split}$}
\]%
\[%
\resizebox{0.99 \columnwidth}{!}{$
\begin{split}
\dot{\eta}_{j}^{y}(t)=  &  -\frac{\partial H}{\partial s_{j}^{y}}\\
=  &  -\Bigg[\frac{1}{TM_{I}}\frac{\partial I_{j}(t)}{\partial s_{j}^{y}%
}+\frac{1}{TM_{R}}\sum_{j}\int_{S}\big(R(w,t)+R_{\!{Bj}}(w,t)\big)\frac
{\partial P_{j}(w,t)}{\partial s_{j}^{y}}dw\\
&  +\sum_{i}\frac{\partial}{\partial s_{j}^{y}}\lambda_{i}(t)\dot{X}%
_{i}(t)+\sum_{i}\frac{\partial}{\partial s_{j}^{y}}\gamma_{i}(t)\dot{Y}%
_{i}(t)+\sum_{i}\frac{\partial}{\partial s_{j}^{y}}\phi_{ij}(t)\dot{Z}%
_{ij}(t)\Bigg]
\end{split}$}
\]%
\[
\eta_{j}^{x}(T)=\eta_{j}^{y}(T)=0
\]
}
%where $\mathcal{R}(s_j(t))=\Bigg{w_i | p_{ij}(t)>0 \Bigg}$ and $\mathcal{R}^B(s_j(t))=\Bigg{w_i | Z_{ij}>0 \mbox{ and } p_{\!_{Bj}}(t)>0 \Bigg}$
From \eqref{hamiltonian}, after some trigonometric manipulations, we get
\begin{equation}
\resizebox{0.89 \columnwidth}{!}{$
\begin{split}
&  H(\mathbf{X},{\boldsymbol{\lambda}},\mathbf{u},{\boldsymbol{\theta}}%
)=\frac{1}{T}\Big[qJ_{1}(t)-(1-q)J_{2}(t)+J_{3}(t)+J_{4}(t)\Big]\\
&  +\sum_{i}\lambda_{i}(t)\dot{X}_{i}(t)+\sum_{i}\gamma_{i}(t)\dot{Y}%
_{i}(t)+\sum_{i}\sum_{j}\phi_{ij}(t)\dot{Z}_{ij}(t)\\
&  +\sum_{j}u_{j}(t)\mbox{sgn}{\eta_{j}^{y}(t)}\sqrt{{\eta_{j}^{x}(t)}%
^{2}+{\eta_{j}^{y}(t)}^{2}}\sin(\theta_{j}(t)+\psi_{j}(t))
\end{split}$}
\label{HamiltonianU}%
\end{equation}
where $\tan\psi_{j}(t)=\frac{\eta_{j}^{x}(t)}{\eta_{j}^{y}(t)}$ for $\eta
_{j}^{y}(t)\neq0$ and $\psi_{j}(t)=\mbox{sgn}{\eta_{j}^{x}(t)}\frac{\pi}{2}$
if $\eta_{j}^{y}(t)=0$. Applying the Pontryagin principle to
\eqref{hamiltonian} with $(\mathbf{u}^{\ast},{\boldsymbol{\theta}}^{\ast})$
being the optimal control, we have:
\begin{equation}
H(\mathbf{X}^{\ast},{\boldsymbol{\lambda}}^{\ast},\mathbf{u}^{\ast
},{\boldsymbol{\theta}}^{\ast})=\min\limits_{\mathbf{u}(t),\boldsymbol{\theta
}(t)}H(\mathbf{X},{\boldsymbol{\lambda}},\mathbf{u},{\boldsymbol{\theta}})
\end{equation}
From \eqref{HamiltonianU} we see that we can always set the control
$\theta_{j}(t)$ to ensure that $\mbox{sgn}{\eta_{j}^{y}(t)}\sin(\theta
_{j}(t)+\psi_{j}(t))<0$. Hence, recalling that $0\leq u_{j}(t)\leq1$,
\begin{equation}
u_{j}^{\ast}(t)=1 \label{OptimU}%
\end{equation}
and
\begin{align}
\sin(\theta_{j}^{*}(t)+\psi_{j}(t))=1  &  \mbox{  if } \mu_{j}^{y}%
(t)<0\nonumber\\
\sin(\theta_{j}^{*}(t)+\psi_{j}(t))=-1  &  \mbox{  if } \mu_{j}^{y}(t)>0
\end{align}
Following the Hamiltonian definition in \eqref{hamiltonian} we have:
\begin{equation}
\frac{\partial H}{\partial\theta_{j}}=-\eta_{j}^{x}(t)u_{j}(t)\sin\theta
_{j}(t)+\eta_{j}^{y}(t)u_{j}(t)\cos\theta_{j}(t) \label{dHdtheta}%
\end{equation}
and setting $\frac{\partial H}{\partial\theta_{j}}=0$ the optimal heading
$\theta_{j}^{\ast}(t)$ should satisfy:%

\begin{equation}
\tan\theta_{j}^{\ast}(t)=\frac{{\eta_{j}^{y}}(t)}{{\eta_{j}^{x}}(t)}%
\end{equation}

Since $u_{j}^{\ast}(t)=1$, we only need to evaluate $\theta_{j}^{\ast}(t)$ for
all $t\in\lbrack0,T]$. This is accomplished by discretizing the problem in
time and numerically solving a TPBVP with a forward integration of the state
and a backward integration of the costate. Solving this problem becomes
intractable as the number of agents and targets grows. The fact that we are dealing with a hybrid dynamic system further complicates
the solution of a TPBVP. On the other hand, it enables us to make use of
Infinitesimal Perturbation Analysis (IPA) \cite{Cassandras2010} to carry out
the parametric trajectory optimization process discussed in the next section.
In particular, we propose a parameterization of agent trajectories allowing us
to utilize IPA to obtain an unbiased estimate for the objective function
gradient with respect to the trajectory parameters.

\subsection{Agent Trajectory Parameterization and Optimization}

\label{parametric} The key idea is to represent each agent's trajectory
through general parametric equations
\begin{equation}%
\begin{array}
[c]{ll}%
s_{j}^{x}(t)=f(\Theta_{j},\rho_{j}(t)),\text{ \ }\quad s_{j}^{y}%
(t)=g(\Theta_{j},\rho_{j}(t)) &
\end{array}
\label{param_traj}%
\end{equation}
where the function $\rho_{j}(t)$ controls the position of the agent on its
trajectory at time $t$ and $\Theta_{j}$ is a vector of parameters controlling
the shape and location of the agent $j$ trajectory. Let $\Theta=[\Theta
_{1},\dots,\Theta_{N}]$. We now replace problem $\mathbf{P1}$ in
(\ref{GenOptim}) by problem $\mathbf{P2}$:
\begin{equation}%
\begin{split}
\min\limits_{\Theta\in F_{\Theta}}  &  \frac{1}{T}\int_{0}^{T}\Big[qJ_{1}%
(\Theta,t)-(1-q)J_{2}(\Theta,t)+J_{3}(\Theta,t)\\
&  +J_{4}(\Theta,t)\Big]dt+J_{f}(\Theta,T)
\end{split}
\label{ParamOptim}%
\end{equation}
where we return to allowing arbitrary stochastic data arrival processes
$\{\sigma_{i}(t)\}$ so that $\mathbf{P2}$ is a parametric stochastic
optimization problem with the feasible parameter set $F_{\Theta}$
appropriately defined depending on \eqref{param_traj}. The cost function in
(\ref{ParamOptim}) is written as%
\[
J(\Theta,T;\mathbf{X}(\Theta,0))=E[\mathcal{L}(\Theta,T;\mathbf{X}%
(\Theta,0))]
\]
where $\mathcal{L}(\Theta,T;\mathbf{X}(\Theta,0))$ is a sample function
defined over $[0,T]$ and $\mathbf{X}(\Theta,0)$ is the initial value of the
state vector. For convenience, in the sequel we will use $\mathcal{L}_{i}$,
$i=1,\ldots,4$, and $\mathcal{L}_{f}$ to denote sample functions of $J_{i}$,
$i=1,\ldots,4$, and $J_{f}$ respectively. Note that in (\ref{ParamOptim}) we
suppress the dependence of the four objective function components on the
controls $\mathbf{u}(t)$ and $\boldsymbol{\theta}(t)$ and stress instead their
dependence on the parameter vector $\Theta$.

In the rest of the paper, we will consider two families of trajectories
motivated by a similar approach used in the multi-agent persistent monitoring
problem in \cite{Lin2015}: \emph{elliptical} trajectories and a more general
\emph{Fourier series} trajectory representation better suited for non-uniform
target topologies. The hybrid dynamics of the data harvesting system allow us
to apply the theory of IPA \cite{Cassandras2010} to obtain on line the
gradient of the sample function $\mathcal{L}(\Theta,T;\mathbf{X}(\Theta,0))$
with respect to $\Theta$. The value of the IPA\ approach is twofold: $(i)$ The
sample gradient $\nabla\mathcal{L}(\Theta,T)$ can be obtained on line based on
observable sample path data \emph{only}, and $(ii)$ $\nabla\mathcal{L}%
(\Theta,T)$ is an unbiased estimate of $\nabla J(\Theta,T)$ under mild
technical conditions as shown in \cite{Cassandras2010}. Therefore, we can use
$\nabla\mathcal{L}(\Theta,T)$ in a standard gradient-based stochastic
optimization algorithm
\begin{equation}
\Theta^{l+1}=\Theta^{l}-\boldsymbol{\nu}_{l}\nabla\mathcal{L}(\Theta
^{l},T),\text{ \ }l=0,1,\ldots\label{SAalgo}%
\end{equation}
to converge (at least locally) to an optimal parameter vector $\Theta^{\ast}$
with a proper selection of a step-size sequence $\{\boldsymbol{\nu}_{l}\}$
\cite{Kushner2003}. We emphasize that this process is carried out \emph{on
line}, i.e., the gradient is evaluated by observing a trajectory with given
$\Theta$ over $[0,T]$ and is iteratively adjusted until convergence is attained.

\subsubsection{IPA Calculus Review and Implementation}

Based on the events defined earlier, we will specify event time derivative and
state derivative dynamics for each mode of the hybrid system. In this process,
we will use the IPA notation from \cite{Cassandras2010} so that ${\tau_{k}}$
is the $k$th event time in an observed sample path of the hybrid system and
${\tau_{k}^{\prime}}=\frac{d\tau_{k}}{d\Theta}$, $\mathcal{X}^{\prime
}(t)=\frac{d\mathcal{X}}{d\Theta}$ are the Jacobian matrices of partial
derivatives with respect to all components of the controllable parameter
vector $\Theta$. Throughout the analysis we will be using $(\cdot)^{\prime}$
to show such derivatives. We will also use $f_{k}(t)=\frac{d\mathcal{X}}{dt}$
to denote the state dynamics in effect over an interevent time interval
$[{\tau_{k},\tau_{k+1})}$. We review next the three fundamental IPA equations
from \cite{Cassandras2010} based on which we will proceed.

First, events may be classified as exogenous or endogenous. An event is
exogenous if its occurrence time is independent of the parameter $\Theta$,
hence ${\tau_{k}^{\prime}}=0$. Otherwise, an endogenous event takes place when
a condition $g_{k}(\Theta,\mathcal{X})=0$ is satisfied, i.e., the state
$\mathcal{X}(t)$ reaches a switching surface described by $g_{k}%
(\Theta,\mathcal{X})$. In this case, it is shown in \cite{Cassandras2010} that%
\begin{equation}
{\tau_{k}^{\prime}}=-\Big(\frac{dg_{k}}{d\mathcal{X}}f_{k}(\tau_{k}%
^{-})\Big)^{-1}\Big(g^{\prime}_{k}+\frac{dg_{k}}{d\mathcal{X}}{\mathcal{X}%
^{\prime}(\tau_{k}^{-})}\Big) \label{tauprime}%
\end{equation}
as long as $\frac{\partial g_{k}}{\partial\mathcal{X}}f_{k}(\tau_{k}^{-}%
)\neq0$. It is also shown in \cite{Cassandras2010} that the state derivative
$\mathcal{X}^{\prime}(t)$ satisfies
\begin{equation}
\frac{d}{dt}\mathcal{X}^{\prime}(t)=\frac{df_{k}}{d\mathcal{X}}\mathcal{X}%
^{\prime}(t)+f^{\prime}_{k}(t),\text{ \ \ }t\in\lbrack{\tau_{k},\tau_{k+1})}
\label{dXdtprime}%
\end{equation}
\begin{equation}
\mathcal{X}^{\prime}(\tau_{k}^{+})=\mathcal{X}^{\prime}(\tau_{k}^{-}%
)+[f_{k-1}(\tau_{k}^{-})-f_{k}(\tau_{k}^{+})]{\tau_{k}}^{\prime}
\label{Xprime}%
\end{equation}
Then, $\mathcal{X}^{\prime}(t)$ for $t\in\lbrack{\tau_{k},\tau_{k+1})}$ is
calculated through
\begin{equation}
\mathcal{X}^{\prime}(t)=\mathcal{X}^{\prime}(\tau_{k}^{+})+\int_{\tau_{k}}%
^{t}\frac{d}{dt}\mathcal{X}^{\prime}(t)dt \label{Xprime_t}%
\end{equation}
Table \ref{eventlist} contains all possible \emph{endogenous} event types for
our hybrid system. To these, we add \emph{exogenous} events $\kappa_{i}$,
$i=1,...,M$, to allow for possible discontinuities (jumps) in the random
processes $\{\sigma_{i}(t)\}$ which affect the sign of $\sigma_{i}(t)-\mu
_{ij}p_{ij}(t)$ in (\ref{Xdot}). We will use the notation $e(\tau_{k})$ to
denote the event type occurring at $t=\tau_{k}$ with $e(\tau_{k})\in E$, the
event set consisting of all endogenous and exogenous events. Finally, we make
the following assumption which is needed in guaranteeing the unbiasedness of
the IPA gradient estimates: $(\mathbf{A6})$ Two events occur at the same time
w.p. $0$ unless one is directly caused by the other.

\subsubsection{Objective Function Gradient}

The sample function gradient $\nabla\mathcal{L}(\Theta,T)$ needed in
(\ref{SAalgo}) is obtained from (\ref{ParamOptim}) assuming a total number of
$K$ events over $[0~T]$ with $\tau_{\!_{K+1}}=T$ and $\tau_{0}=0$:
\begin{equation}
\resizebox{0.89 \columnwidth}{!}{$
\begin{split}
&  \nabla\mathcal{L}(\Theta,T;\mathbf{X}(\Theta;0))=\label{gradJ}\\
&  \quad\frac{1}{T}\nabla\Big[\int_{0}^{T}\Big(q\mathcal{L}_{1}(\Theta
,t)-(1-q)\mathcal{L}_{2}(\Theta,t)+\mathcal{L}_{3}(\Theta,t)\\
&  \quad+\mathcal{L}_{4}(\Theta,t)\Big)dt\Big]+\nabla\mathcal{L}_{f}%
(\Theta,T)\\
&  =\frac{1}{T}\nabla\Big[\sum\limits_{k=0}^{K}\int_{\tau_{k}}^{\tau_{k+1}%
}\Big(q\mathcal{L}_{1}(\Theta,t)-(1-q)\mathcal{L}_{2}(\Theta,t)+\mathcal{L}%
_{3}(\Theta,t)\\
&  \quad+\mathcal{L}_{4}(\Theta,t)\Big)dt\Big]+\nabla\mathcal{L}_{f}%
(\Theta,T)\\
&  =\frac{1}{T}\Big[\sum\limits_{k=0}^{K}q\Big(\int_{\tau_{k}}^{\tau_{k+1}%
}\nabla\mathcal{L}_{1}(\Theta,t)dt+\mathcal{L}_{1}(\Theta,\tau_{k+1}%
)\tau_{k+1}^{\prime}-\mathcal{L}_{1}(\Theta,\tau_{k})\tau_{k}^{\prime
}\Big)\\
&  \quad-(1-q)\Big(\int_{\tau_{k}}^{\tau_{k+1}}\nabla\mathcal{L}_{2}%
(\Theta,t)dt+\mathcal{L}_{2}(\Theta,\tau_{k+1})\tau_{k+1}^{\prime}%
-\mathcal{L}_{2}(\Theta,\tau_{k})\tau_{k}^{\prime}\Big)\\
&  \quad+\Big(\int_{\tau_{k}}^{\tau_{k+1}}\nabla\mathcal{L}_{3}(\Theta
,t)dt+\mathcal{L}_{3}(\Theta,\tau_{k+1})\tau_{k+1}^{\prime}-\mathcal{L}%
_{3}(\Theta,\tau_{k})\tau_{k}^{\prime}\Big)\\
&  \quad+\Big(\int_{\tau_{k}}^{\tau_{k+1}}\nabla\mathcal{L}_{4}(\Theta
,t)dt+\mathcal{L}_{4}(\Theta,\tau_{k+1})\tau_{k+1}^{\prime}-\mathcal{L}%
_{4}(\Theta,\tau_{k})\tau_{k}^{\prime}\Big)\Big]\\
&  \quad+\nabla\mathcal{L}_{f}(\Theta,T)\\
&  =\frac{1}{T}\Big[\sum\limits_{k=0}^{K}\int_{\tau_{k}}^{\tau_{k+1}%
}\Big(q\nabla\mathcal{L}_{1}(\Theta,t)-(1-q)\nabla\mathcal{L}_{2}%
(\Theta,t)+\nabla\mathcal{L}_{3}(\Theta,t)\\
&  \quad+\nabla\mathcal{L}_{4}(\Theta,t)\Big)dt\Big]+\nabla\mathcal{L}%
_{f}(\Theta,T)
\end{split}
$}
\end{equation}
The last step follows from the continuity of the state variables which causes
adjacent limit terms in the sum to cancel out. Therefore, $\nabla
\mathcal{L}(\Theta,T)$ does not have any direct dependence on any $\tau
_{k}^{\prime}$; this dependence is indirect through the state derivatives
involved in the four individual gradient terms.

Referring to (\ref{J1}), the first term in \eqref{gradJ} involves
$\nabla\mathcal{L}_{1}(\Theta,t)$ which is as a sum of $X_{i}^{\prime}(t)$
derivatives. Similarly, $\nabla\mathcal{L}_{2}(\Theta,t)$ is a sum of
$Y_{i}^{\prime}(t)$ derivatives and $\nabla\mathcal{L}_{f}(\Theta,T)$ requires
only $Z_{ij}^{\prime}(T)$. The third term, $\nabla\mathcal{L}_{3}(\Theta,t)$,
requires derivatives of $I_{j}(t)$ in (\ref{Ij}) which depend on the
derivatives of the max function in \eqref{Dplus} and the agent state
derivatives $s_{j}^{\prime}(t)$ with respect to $\Theta$. The term
$\nabla\mathcal{L}_{4}(\Theta,t)$ needs the values of $X_{i}^{\prime}(t)$ and
$Z_{ij}^{\prime}(t)$. The gradients of the last two terms are derived in the
appendix. Possible discontinuities in these derivatives occur when any of the
last four events in Table \ref{eventlist} takes place.

In summary, the evaluation of \eqref{gradJ} requires the state derivatives
$X_{i}^{\prime}(t)$, $Z_{ij}^{\prime}(t)$, $Y_{i}^{\prime}(t)$, and
$s_{j}^{{\prime}}(t)$. The latter are easily obtained for any specific choice
of $f$ and $g$ in (\ref{param_traj}) and are shown in Appendix
\ref{AppEllipse}. The former require a rather laborious use of (\ref{tauprime}%
)-(\ref{Xprime}) which, however, reduces to a simple set of state derivative
dynamics as shown next.

\begin{proposition}: After an event occurrence at $t=\tau_{k}$, the state
derivatives $X_{i}^{\prime}(\tau_{k}^{+})$, $Y_{i}^{\prime}(\tau_{k}^{+})$,
$Z_{ij}^{\prime}(\tau_{k}^{+})$, with respect to the controllable parameter
$\Theta$ satisfy the following:%
\[
X_{i}^{\prime}(\tau_{k}^{+})=\left\{
\begin{array}
[c]{ll}%
0 & \text{if }e(\tau_{k})=\xi_{i}^{0}\\
X_{i}^{\prime}(\tau_{k}^{-})-\mu_{il}p_{il}(\tau_{k}){\tau_{k}^{\prime}} &
\text{if }e(\tau_{k})=\delta_{ij}^{+}\\
X_{i}^{\prime}(\tau_{k}^{-}) & \text{otherwise}%
\end{array}
\right.
\]
where $l\neq j$ with $p_{il}(\tau_{k})>0$ if such $l$ exists and ${\tau
_{k}^{\prime}=}\frac{\partial d_{ij}(s_{j})}{\partial s_{j}}s_{j}^{\prime
}\left(  \frac{\partial d_{ij}(s_{j})}{\partial s_{j}}\dot{s}_{j}(\tau
_{k})\right)  ^{-1}$.
\begin{align*}
Y_{i}^{\prime}(\tau_{k}^{+})  &  =\left\{
\begin{array}
[c]{ll}%
Y_{i}^{\prime}(\tau_{k}^{-})+Z_{ij}^{\prime}(\tau_{k}^{-}) & \text{if }%
e(\tau_{k})=\zeta_{ij}^{0}\\
Y_{i}^{\prime}(\tau_{k}^{-}) & \text{otherwise}%
\end{array}
\right. \\
Z_{ij}^{\prime}(\tau_{k}^{+})  &  =\left\{
\begin{array}
[c]{ll}%
0 & \text{if }e(\tau_{k})=\zeta_{ij}^{0}\\
Z_{ij}^{\prime}(\tau_{k}^{-})+X_{i}^{\prime}(\tau_{k}^{-}) & \text{if }%
e(\tau_{k})=\xi_{i}^{0}\\
Z_{ij}^{\prime}(\tau_{k}^{-}) & \text{otherwise}%
\end{array}
\right.
\end{align*}
where $e(\tau_{k})=\xi_{i}^{0}$ occurs when $j$ is connected to target $i$.
\end{proposition}
\begin{proof}: See \eqref{Xprime0}, \eqref{xiplus}, \eqref{Xplusjump},
\eqref{Yprime}, \eqref{yrplus}, \eqref{zijplus}, \eqref{Zprime0},
\eqref{Zprime01} in Appendix \ref{IPAapp}.
\end{proof}

This result shows that only three of the events in $E$ can actually cause
discontinuous changes to the state derivatives. Further, note that
$X_{i}^{\prime}(t)$ is reset to zero after a $\xi_{i}^{0}$ event. Moreover,
when such an event occurs, note that $Z_{ij}^{\prime}(t)$ is coupled to
$X_{i}^{\prime}(t)$. Similarly for $Z^{\prime}_{ij}(t)$ and $Y^{\prime}_{i}(t)
$ when event $\zeta_{ij}^{0}$ occurs, showing that perturbations in $\Theta$
can only propagate to an adjacent queue when that queue is emptied.

\begin{proposition}: The state derivatives $X_{i}^{\prime}(\tau_{k+1}^{-})$,
$Y_{i}^{\prime}(\tau_{k+1}^{-})$ with respect to the controllable parameter
$\Theta$ satisfy the following after an event occurrence at $t=\tau_{k}$:%
\begin{align*}
X_{i}^{\prime}(\tau_{k+1}^{-})  &  =\left\{
\begin{array}
[c]{ll}%
0 & \text{if }e(\tau_{k})=\xi_{i}^{0}\\
X_{i}^{\prime}(\tau_{k}^{+})-\int_{\tau_{k}}^{\tau_{k+1}}\mu_{ij}%
p_{ij}^{\prime}(u)du & \text{otherwise}%
\end{array}
\right. \\
Y_{i}^{\prime}(\tau_{k+1}^{-})  &  =Y_{i}^{\prime}(\tau_{k}^{+})+\int
_{\tau_{k}}^{\tau_{k+1}}\beta_{i}^{\prime}(u)du
\end{align*}
where $j$ is such that $p_{ij}(t)>0$, $t\in\lbrack{\tau_{k},\tau_{k+1})}$.
\end{proposition}
\begin{proof}: See \eqref{Xprimet1}, \eqref{dxdt} and \eqref{dydt} in Appendix
\ref{IPAapp}.
\end{proof}
\begin{proposition}: The state derivatives $Z_{ij}^{\prime}(\tau_{k+1}^{+})$
with respect to the controllable parameter $\Theta$ satisfy the following
after an event occurrence at $t=\tau_{k}$:\newline\emph{i}- If $j$ is
connected to target $i$,%
\[
Z_{ij}^{\prime}(\tau_{k+1}^{-})=\left\{
\begin{array}
[c]{ll}%
Z_{ij}^{\prime}(\tau_{k}^{+})\qquad\qquad\text{if }e(\tau_{k})=\xi_{i}%
^{0},\text{ }\zeta_{ij}^{0}\text{ or }\delta_{ij}^{+} & \\
Z_{ij}^{\prime}(\tau_{k}^{+})+\int_{\tau_{k}}^{\tau_{k+1}}\mu_{ij}%
p_{ij}^{\prime}(u)du\quad\text{otherwise} &
\end{array}
\right.
\]
\emph{ii}- If $j$ is connected to $B$ with $Z_{ij}(\tau_{k})>0$,
\[
Z_{ij}^{\prime}(\tau_{k+1}^{-})=Z_{ij}^{\prime}(\tau_{k}^{+})-\int_{\tau_{k}%
}^{\tau_{k+1}}\beta_{ij}p_{Bj}^{\prime}(u)du
\]
\emph{iii}- Otherwise, $Z_{ij}^{\prime}(\tau_{k+1}^{-})=Z_{ij}^{\prime}%
(\tau_{k}^{+})$.
\end{proposition}
\begin{proof}: See \eqref{dzdt}, \eqref{dzdt1}, \eqref{Zprimet0} and
\eqref{ZprimetB} in Appendix \ref{IPAapp}.
\end{proof}
\begin{corollary}
\label{corollary1} The state derivatives $X_{i}^{\prime}(t)$, $Z_{ij}^{\prime
}(t)$, $Y_{i}^{\prime}(t)$ with respect to the controllable parameter $\Theta$
are independent of the random data arrival processes $\{\sigma_{i}(t)\}$,
$i=1,\ldots,M$.
\end{corollary}

\begin{proof}: Follows directly from the three Propositions.
\end{proof}

There are a few important consequences of these results. First, as the
Corollary asserts, one can apply IPA regardless of the characteristics of the
random processes $\{\sigma_{i}(t)\}$. This robustness property does not mean
that these processes do not affect the values of the $X_{i}^{\prime}(t)$,
$Z_{ij}^{\prime}(t)$, $Y_{i}^{\prime}(t)$; this happens through the values of
the event times ${\tau_{k}}$, $k=1,2,\ldots$, which are observable and enter
the computation of these derivatives as seen above.

Second, the IPA estimation process is event-driven: $X_{i}^{\prime}(\tau
_{k}^{+})$, $Y_{i}^{\prime}(\tau_{k}^{+})$, $Z_{ij}^{\prime}(\tau_{k}^{+})$
are evaluated at event times and then used as initial conditions for the
evaluations of $X_{i}^{\prime}(\tau_{k+1}^{-})$, $Y_{i}^{\prime}(\tau
_{k+1}^{-})$, $Z_{ij}^{\prime}(\tau_{k+1}^{-})$ along with the integrals
appearing in Propositions 2,3 which can also be evaluated at $t=\tau_{k+1}$.
Consequently, this approach is scalable in the number of events in the system
as the number of agents and targets increases.

Third, despite the elaborate derivations in the Appendix, the actual
implementation reflected by the three Propositions is simple. Finally,
returning to (\ref{gradJ}), note that the integrals involving $\nabla
\mathcal{L}_{1}(\Theta,t)$, $\nabla\mathcal{L}_{2}(\Theta,t)$ are directly
obtained from $X_{i}^{\prime}(t)$, $Y_{i}^{\prime}(t)$, the integral involving
$\nabla\mathcal{L}_{3}(\Theta,t)$ is obtained from straightforward
differentiation of (\ref{Ij}), and the final term is obtained from
$Z_{ij}^{\prime}(T)$.

\subsubsection{Objective Function Optimization}

This is carried out using (\ref{SAalgo}) with an appropriate diminishing step
size sequence.

\textbf{Elliptical Trajectories:} Elliptical trajectories are described by
their center coordinates, minor and major axes and orientation. Agent $j$'s
position $s_{j}(t)=[s_{j}^{x}(t),s_{j}^{y}(t)]$ follows the general parametric
equation of the ellipse:%
\begin{equation}%
\resizebox{0.89 \columnwidth}{!}{$
\begin{array}
[c]{ll}%
s_{j}^{x}(t)= & A_{j}+a_{j}\cos\rho_{j}(t)\cos\phi_{j}-b_{j}\sin\rho
_{j}(t)\sin\phi_{j}\\
s_{j}^{y}(t)= & B_{j}+a_{j}\cos\rho_{j}(t)\sin\phi_{j}+b_{j}\sin\rho
_{j}(t)\cos\phi_{j}%
\end{array}$}
\end{equation}
Here, $\Theta_{j}=[A_{j},B_{j},a_{j},b_{j},\phi_{j}]$ where $A_{j},B_{j}$ are
the coordinates of the center, $a_{j}$ and $b_{j}$ are the major and minor
axis respectively while $\phi_{j}\in\lbrack0,\pi)$ is the ellipse orientation
which is defined as the angle between the $x$ axis and the major axis of the
ellipse. The time dependent parameter $\rho_{j}(t)$ is the eccentric anomaly
of the ellipse. Since the agent is moving with constant speed of 1 on this
trajectory from \eqref{OptimU}, we have $\dot{s}_{j}^{x}(t)^{2}+\dot{s}%
_{j}^{y}(t)^{2}=1$ which gives%
\begin{equation*}
\resizebox{0.99 \columnwidth}{!}{$
\dot\rho_{j}(t)=\left[
\begin{array}
[c]{ll}
& \Big(a\sin\rho_{j}(t)\cos\phi_{j}+b_{j}\cos\rho_{j}(t)\sin\phi_{j}%
\Big)^{2}\\
& \quad+\Big(a\sin\rho_{j}(t)\sin\phi_{j}-b_{j}\cos\rho_{j}(t)\cos\phi
_{j}\Big)^{2}%
\end{array}
\right]  ^{-\frac{1}{2}}$}
\end{equation*}

In the data harvesting problem, trajectories that do not pass through the base
are inadmissible since there is no delivery of data. Therefore, we add a
constraint to force the ellipse to pass through $w_{\!_{B}}=[w_{\!_{B}}%
^{x},w_{\!_{B}}^{y}]$ where:%
\begin{equation}%
\begin{split}
w_{\!_{B}}^{x}=  &  A_{j}+a_{j}\cos\rho_{j}(t)\cos\phi_{j}-b_{j}\sin\rho
_{j}(t)\sin\phi_{j}\\
w_{\!_{B}}^{y}=  &  B_{j}+a_{j}\cos\rho_{j}(t)\sin\phi_{j}+b_{j}\sin\rho
_{j}(t)\cos\phi_{j}%
\end{split}
\end{equation}

%from this we can calculate:
%\begin{equation}
%\sin \rho_j(t)=\frac{(w^y_{\!_B}-B_j)\cos\phi_j-(w^x_{\!_B}-A_j)\sin\phi_j}{b}
%\end{equation}
%and
%\begin{equation}
%\cos\rho_j(t)=\frac{(w^x_{\!_B}-A_j)\cos\phi_j+(w^y_{\!_B}-B_j)\sin\phi_j}{a}
%\end{equation}
Using the fact that $\sin^{2}\rho(t)+\cos^{2}\rho(t)=1$ we define a quadratic
constraint term added to $J(\Theta,T;\mathbf{X}(\Theta,0))$ with a
sufficiently large multiplier. This can ensure the optimal path passes through
the base location $w_{\!_{B}}$. We define $\mathcal{C}_{j}(\Theta_{j})$:
\begin{equation}
\mathcal{C}_{j}(\Theta_{j})=\big(1-f_{j}^{1}\cos^{2}\phi_{j}-f_{j}^{2}\sin
^{2}\phi_{j}-f_{j}^{3}\sin2\phi_{j}\big)^{2}%
\end{equation}
where $f_{j}^{1}=\big(\frac{w_{\!_{B}}^{x}-A_{j}}{a_{j}}\big)^{2}%
+\big(\frac{w_{\!_{B}}^{y}-B_{j}}{b_{j}}\big)^{2}$, $f_{j}^{2}=\big(\frac
{w_{\!_{B}}^{x}-A_{j}}{b_{j}}\big)^{2}+\big(\frac{w_{\!_{B}}^{y}-B_{j}}{a_{j}%
}\big)^{2}$, $f_{j}^{3}=\frac{(b_{j}^{2}-a_{j}^{2})(w_{\!_{B}}^{x}%
-A_{j})(w_{\!_{B}}^{y}-B_{j})}{a_{j}^{2}b_{j}^{2}}$.
\begin{figure*}
\begin{center}
\begin{subfigure}[TPBVP Trajectories for Case I\label{CaseIFig1}]{
\includegraphics[width=1.7in,height=1.9in]{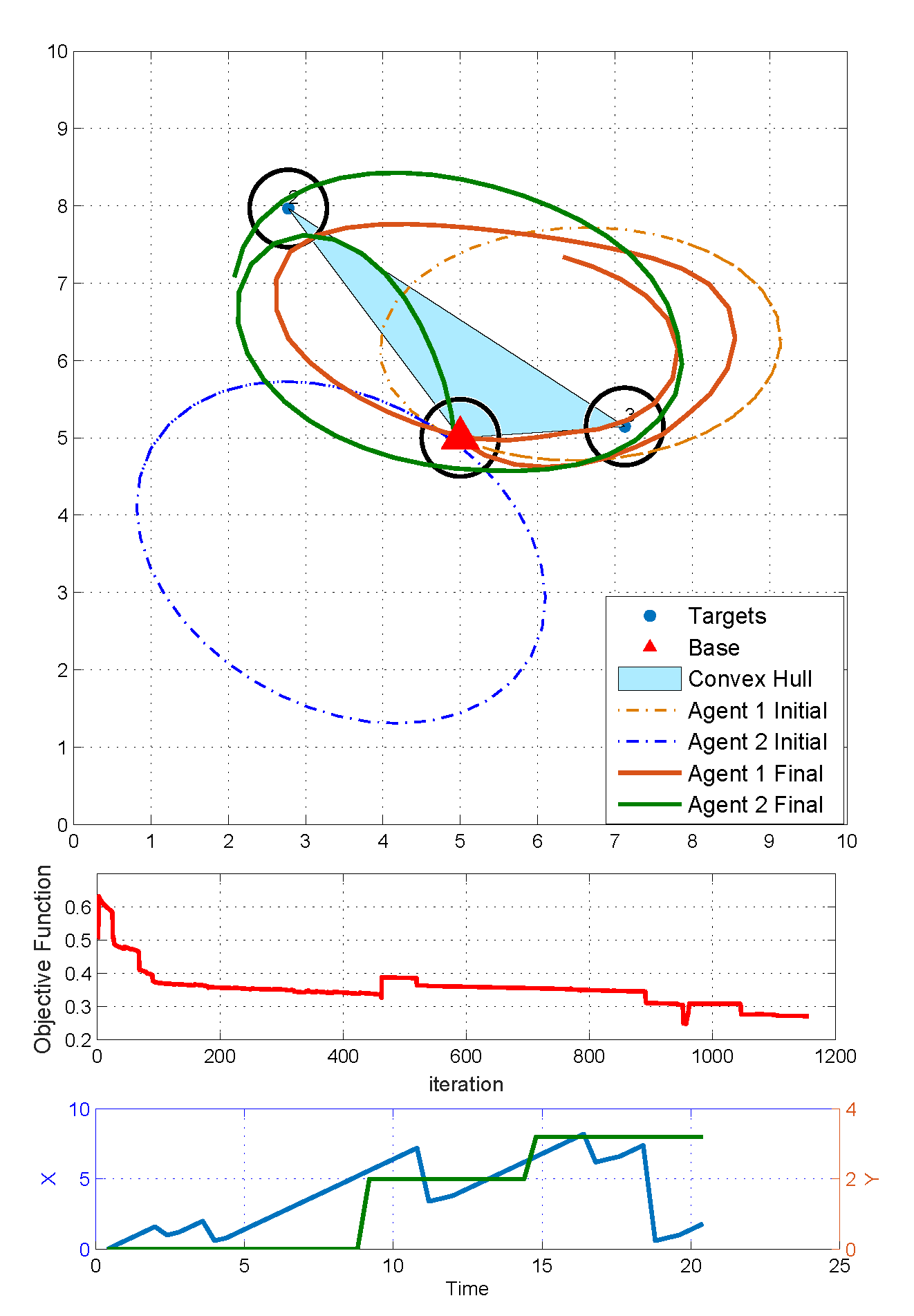}}
\end{subfigure}
\begin{subfigure}[Elliptical Trajectories for case I\label{CaseIFig2}]{
\includegraphics[width=1.7in,height=1.9in]{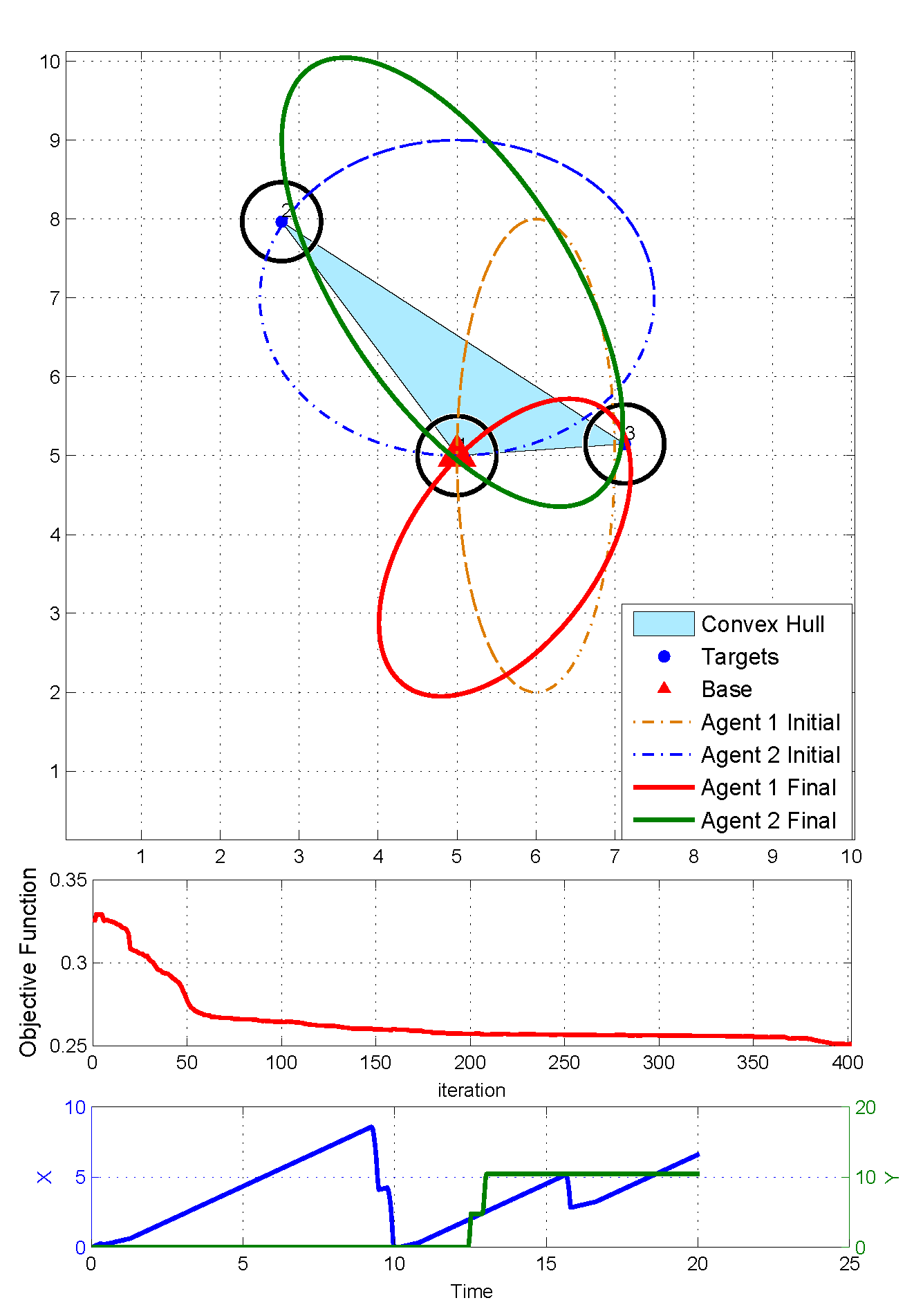}}
\end{subfigure}
\begin{subfigure}[Fourier Trajectories for case I\label{CaseIFig3}]{
\includegraphics[width=1.7in,height=1.9in]{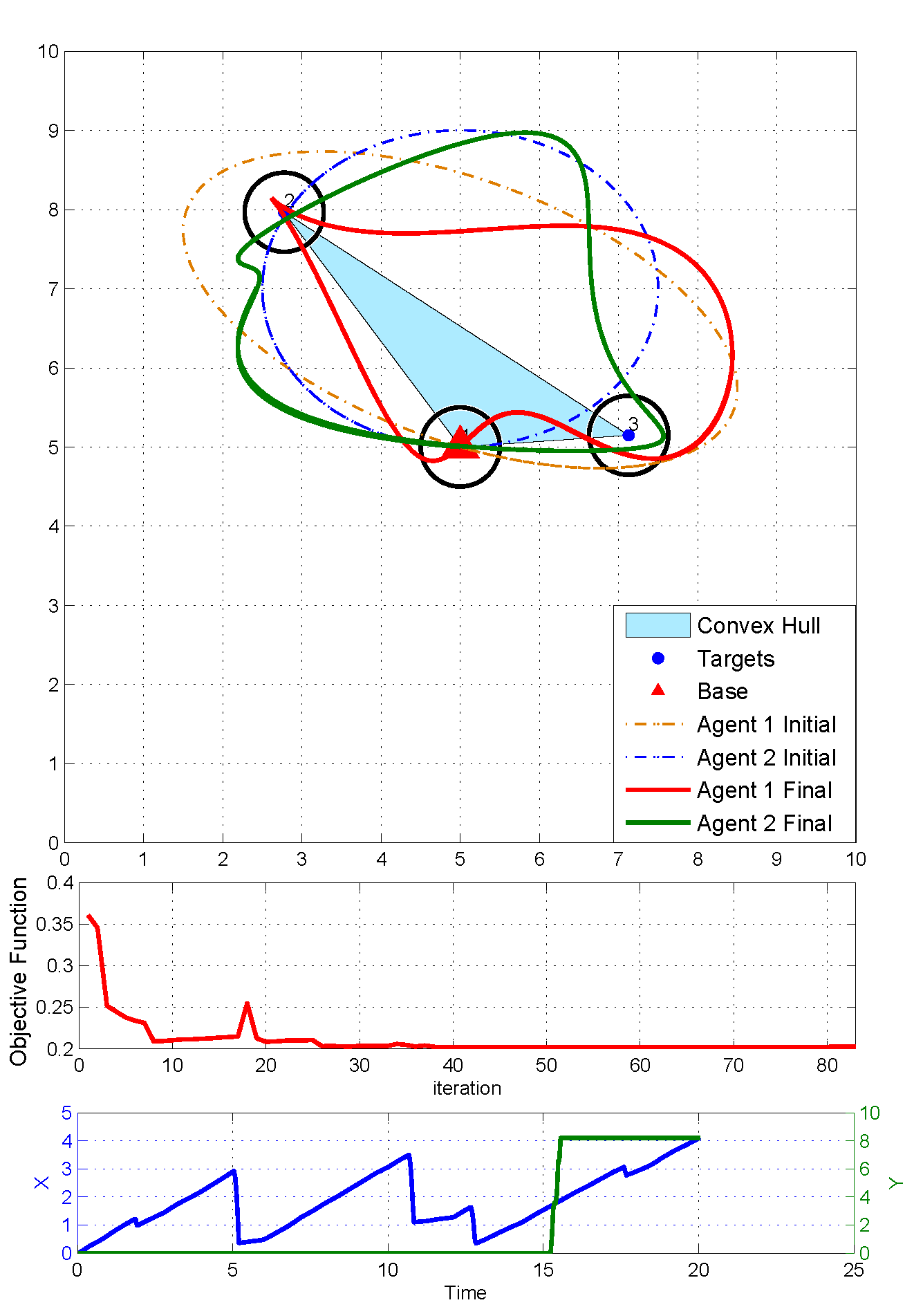}}
\end{subfigure}
\end{center}
\caption{Simulation results for the two-target and two-agent case}
\end{figure*}

Multiple visits to the base may be needed during the mission time $[0,T]$. We
can capture this by allowing an agent trajectory to consist of a sequence of
admissible ellipses. For each agent, we define $\mathcal{E}_{j}$ as the number
of ellipses in its trajectory. The parameter vector $\Theta_{j}^{\kappa}$ with
$\kappa=1,\dots,\mathcal{E}_{j}$, defines the $\kappa th$ ellipse in agent
$j$'s trajectory and $\mathcal{T}_{j}^{\kappa}$ is the time that agent $j$
completes ellipse $\kappa$. Therefore, the location of each agent is described
through $\kappa$ during $[\mathcal{T}_{j}^{\kappa-1},\mathcal{T}_{j}^{\kappa
}]$ where $\mathcal{T}_{j}^{0}=0$. Since we cannot optimize over all possible
$\mathcal{E}_{j}$ for all agents, an iterative process needs to be performed
in order to find the optimal number of segments in each agent's trajectory. At
each step, we fix $\mathcal{E}_{j}$ and find the optimal trajectory with that
many segments. The process is stopped once the optimal trajectory with
$\mathcal{E}_{j}$ segments is no better than the optimal one with
$\mathcal{E}_{j}-1$ segments (obviously, this is generally not a globally optimal
solution). We can now formulate the parametric optimization problem
$\mathbf{P2_{e}}$ where $\Theta_{j}=[\Theta_{j}^{1},\dots,\Theta
_{j}^{\mathcal{E}_{j}}]$ and $\Theta=[\Theta_{1},\dots,\Theta_{N}]$:%
\begin{equation}
\resizebox{0.89 \columnwidth}{!}{$
\begin{split}
\min\limits_{\Theta\in F_{\Theta}}J_{e}=  &  \frac{1}{T}\int_{0}%
^{T}\Big[qJ_{1}(\Theta,t)-(1-q)J_{2}(\Theta,t)+J_{3}(\Theta,t)\\
&  +J_{4}(\Theta,t)\Big]dt+M_{C}\sum\limits_{j=1}^{N}\mathcal{C}_{j}%
(\Theta_{j})+J_{f}(\Theta,T)
\end{split}$}
\label{ParamOptim_e}%
\end{equation}
where $M_{C}$ is a large multiplier. The evaluation of $\nabla\mathcal{C}_{j}$
is straightforward and does not depend on any event (details are shown in
Appendix \ref{AppEllipse}).

\textbf{Fourier Series Trajectories:} The elliptical trajectories are limited
in shape and may not be able to cover many targets in a mission space. Thus,
we next parameterize the trajectories using a Fourier series representation of
closed curves \cite{Zahn1972}. Using a Fourier series function for $f$ and $g$
in \eqref{param_traj}, agent $j$'s trajectory can be described as follows with
base frequencies $f_{j}^{x}$ and $f_{j}^{y}$:
\begin{equation}
\resizebox{0.89 \columnwidth}{!}{$
\begin{array}
[c]{ll}%
s_{j}^{x}(t)= & a_{0,j}+\sum_{n=1}^{\Gamma_{j}^{x}}a_{n,j}%
\sin(2\pi nf_{j}^{x}\rho_{j}(t)+\phi_{n,j}^{x})\\
s_{j}^{y}(t)= & b_{0,j}+\sum_{n=1}^{\Gamma_{j}^{y}}b_{n,j}%
\sin(2\pi nf_{j}^{y}\rho_{j}(t)+\phi_{n,j}^{y})
\end{array}$}
\label{fouriertraj}%
\end{equation}
The parameter $\rho(t)\in\lbrack0,2\pi]$, similar to elliptical trajectories,
represents the position of the agent along the trajectory. In this case,
forcing a Fourier series curve to pass through the base is easier. For
simplicity, we assume a trajectory to start at the base and set $s_{j}%
^{x}(0)=w_{\!_{B}}^{x}$, $s_{j}^{y}(0)=w_{\!_{B}}^{y}$. Assuming $\rho(0)=0$,
with no loss of generality, we can calculate the zero frequency terms by means
of the remaining parameters:
\begin{equation*}
a_{0,j}=w_{\!_{B}}^{x}-\displaystyle\sum_{n=1}^{\Gamma_{j}^{x}}a_{n,j}%
\sin(\phi_{n,j}^{x}),b_{0,j}=w_{\!_{B}}^{y}-\displaystyle\sum_{n=1}%
^{\Gamma_{j}^{y}}b_{n,j}\sin(\phi_{n,j}^{y}) \label{fourier_baseconst}%
\end{equation*}
The parameter vector for agent $j$ is $\Theta_{j}=[f_{j}^{x},a_{0,j}%
,\ldots,a_{\Gamma_{j}^{x}},b_{0,j},\ldots,b_{\Gamma_{j}^{y}},\phi_{1,j}%
,\ldots,\phi_{\Gamma_{j}^{x}},\xi_{1,j},\ldots,\xi_{\Gamma_{j}^{y}}]$ and
$\Theta=[\Theta_{1},\ldots,\Theta_{N}]$. Note that the shape of the curve is
fully captured by the ratio $f_{j}^{x}/f_{j}^{y}$, so that one of these two
parameters can be kept constant. For the Fourier trajectories, the fact that
$\mathbf{u}_{j}^{\ast}=1$ allows us to calculate $\dot{\rho}_{j}(t)$ as
follows:
\begin{equation*}
\resizebox{0.99 \columnwidth}{!}{$
\dot{\rho}_{j}(t)=\frac{1}{2\pi}\left[
\begin{array}
[c]{ll}
& \Bigg(f_{j}^{x}\displaystyle\sum_{n=1}^{\Gamma_{j}^{x}}a_{n,j}n\cos(2\pi
f_{j}^{x}\rho_{j}(t)+\phi_{n,j}^{x})\Bigg)^{2}\\
& \quad+\Bigg(f_{j}^{y}\displaystyle\sum_{n=1}^{\Gamma_{j}^{x}}b_{n,j}%
n\cos(2\pi f_{j}^{y}\rho_{j}(t)+\phi_{n,j}^{y})\Bigg)^{2}%
\end{array}
\right]  ^{-1/2}$}
\end{equation*}
Problem $\mathbf{P2_{f}}$ is the same as $\mathbf{P2}$ but there are no
additional constraints in this case:
\begin{equation}
\resizebox{0.89 \columnwidth}{!}{$
\begin{split}
\min\limits_{\Theta\in F_{\Theta}}J_{f}=  &  \frac{1}{T}\int_{0}%
^{T}\Big[qJ_{1}(\Theta,t)-(1-q)J_{2}(\Theta,t)+J_{3}(\Theta,t)\\
&  +J_{4}(\Theta,t)\Big]dt+J_{f}(T)
\end{split}$}
\label{ParamOptim_f}%
\end{equation}

\section{Numerical Results}

\label{Numerical} In this section numerical results are presented to
illustrate our approach. The mission space $S$ is considered to be
$[0,10]\times\lbrack0,10]$ in all cases. The first case we consider is a small
mission to obtain the TPBVP results and confirm the fact that it is not
scalable to bigger problems.

In Case I we consider a two-target, two-agent setting. We assume deterministic
arrival processes with $\sigma_{i}=0.5$ for all $i$. For \eqref{Pij} and
\eqref{PB} we have used $p(w,v)=\max(0,1-\frac{d(w,v)}{r})$ where $r$ is the
corresponding value of $r_{ij}$ or $r_{\!_{Bj}}$. We set $\mu_{ij}=100$ and
$\beta_{ij}=500$ for all $i$ and $j$. Other parameters used are $q=0.5$,
$r_{ij}=r_{\!_{Bj}}=0.5$ and $T=20$. The trajectory comparison from TPBVP,
Elliptical and Fourier parametric solutions is shown in Figs. \ref{CaseIFig1},
\ref{CaseIFig2}, \ref{CaseIFig3}. In each figure, the trajectories are shown
in the top part, while the actual objective function convergence behavior is
hown in the middle graph. The lower graph shows the total amount of data at
targets at any time (in blue) and the total amount of data at the base (in green).

In the TPBVP results, the main limitation is the number of the time steps in
the discretization of the interval $[0,T]$, since the number of control values
grows with it. To bring this into prospective, for this sample problem with
$T=20$ we considered $300$ time steps, i.e., $300$ values for the heading of
each agent need to be calculated, which brings the total number of controls to
$600$. In contrast, for the same problem the total number of controls for the
elliptical trajectories are $10$ parameters and for the Fourier trajectories
it is $28$. This explains why the TPBVP cannot be a viable solution for larger
values of $T$. Note that, in this scenario the total time is only 20 time
steps so we can obtain a TPBVP solution. This, however, causes a poor
representation for the parameterized trajectories which are spending
significant time outside the convex hull since they have not converged after
only 20 time steps.

In Table \ref{CaseITable}, the actual values for $J^{\ast}$, $J_{1}^{\ast}$,
$J_{2}^{\ast}$ are shown for the three different trajectories of Fig.
\ref{CaseIFig1}.,\ref{CaseIFig2},\ref{CaseIFig3}. Note that the objective is
to minimize $J$ by minimizing $J_{1}-J_{2}$.
\begin{figure}
\begin{subfigure}[Elliptical Trajectories for Case II\label{CaseIIElliptical1}]{
\includegraphics[width=1.5in]{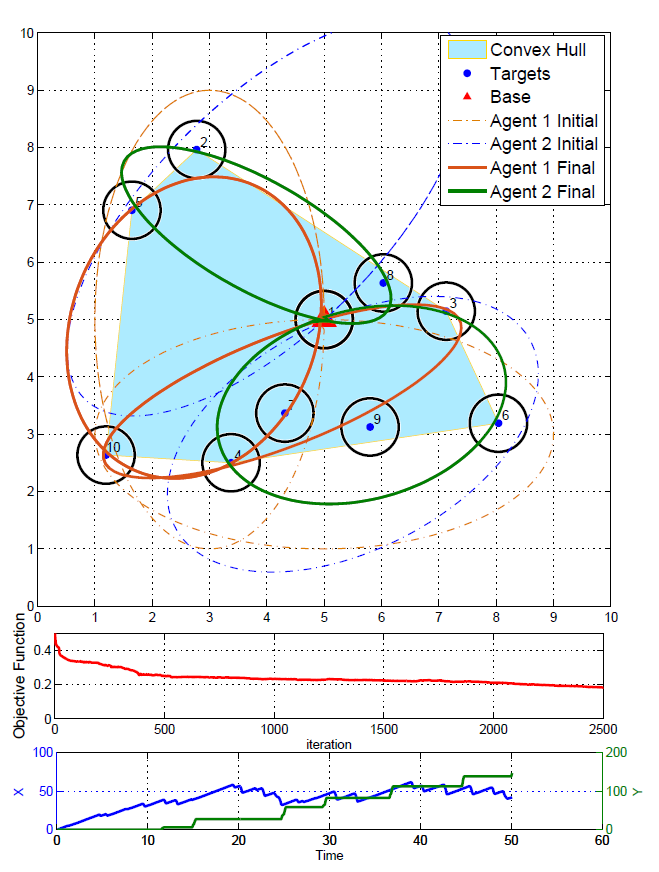}}
\end{subfigure}
\begin{subfigure}[Fourier Trajectories for Case II\label{CaseIIFourier1}]{
\includegraphics[width=1.5in]{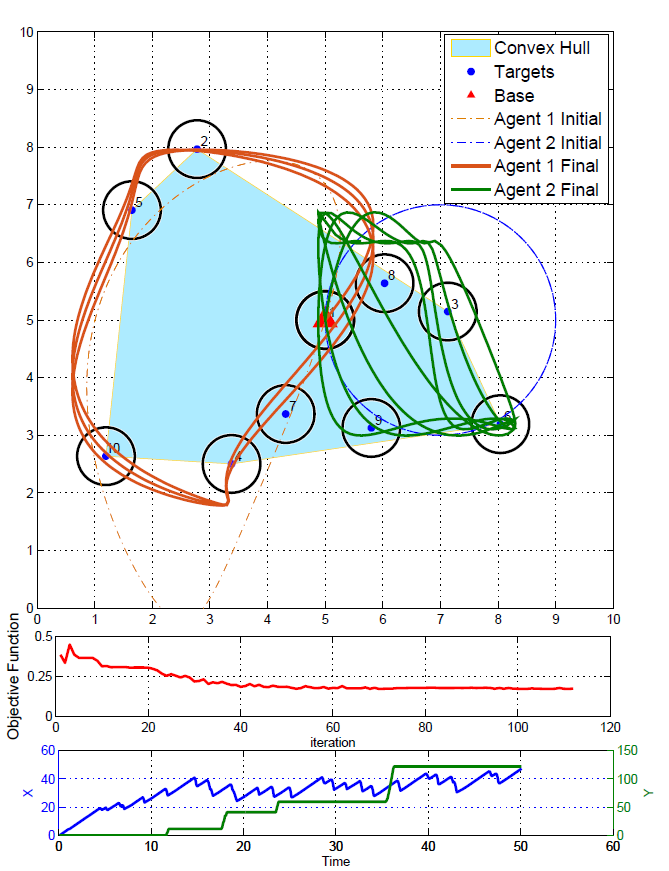}}
\end{subfigure}
\caption{Simulation results for 9-target and two-agent case}
\end{figure}

Next, in Case II we consider 9 targets and 2 agents. The base is located at
the center of the mission space. We have $\sigma_{i}(t)=0.5$, $\mu_{ij}=50$
and $\beta_{ij}=500$ for all $i$ and $j$. Other parameters used are $q=0.5$,
$r_{ij}=0.55,~r_{\!_{Bj}}=0.65$ and $T=50$. In Fig. \ref{CaseIIElliptical1}
the solution with two ellipses in each agent's trajectory is shown. As can be
seen the trajectory correctly finds all the target locations and empties the
target queues periodically. Fig. \ref{CaseIIFourier1} shows the Fourier
trajectories. The two graphs on the bottom show the
objective function value and instantaneous total content at targets and base.
The results for Case II are summarized in Table \ref{CaseIITable}.
\begin{table}
\caption{Result Comparison for Case I}%
\label{CaseITable}
\centering
{\normalsize
\begin{tabular}
[c]{|c|c|c|c|}\hline
\textbf{Method} & $J^{*}$ & $J_{1}^{*}$ & $-J_{2}^{*}$\\\hline
\emph{TPBVP} & 0.272 & 0.098 & -0.038\\\hline
\emph{Elliptical} & 0.255 & 0.092 & -0.095\\\hline
\emph{Fourier} & 0.202 & 0.089 & -0.095\\\hline
\end{tabular}
}\end{table}
\begin{table}
\caption{Results Comparison for Case II}%
\label{CaseIITable}
\centering
{\normalsize \
\begin{tabular}
[c]{|c|c|c|c|}\hline
\textbf{Method} & $J^{*}$ & $J_{1}^{*}$ & $-J_{2}^{*}$\\\hline
\emph{Elliptical} & 0.19 & 0.090 & -0.124\\\hline
\emph{Fourier} & 0.18 & 0.069 & -0.117\\\hline
\end{tabular}
}\end{table}
\begin{table}
\caption{Results Comparison for Case III}%
\label{CaseIIITable}
\centering
{\normalsize \
\begin{tabular}
[c]{|c|c|c|c|}\hline
\textbf{Method} & $J^{*}$ & $J_{1}^{*}$ & $-J_{2}^{*}$\\\hline
\emph{Elliptical} & 0.35 & 0.12 & -0.09\\\hline
\emph{Fourier} & 0.23 & 0.09 & -0.1\\\hline
\emph{Fourier (Stochastic Arrival)} & 0.23 & 0.13 & -0.13\\\hline
\end{tabular}
}\end{table}

Case III has 12 targets that are uniformly distributed in the
mission space. Here, we try to examine the robustness of our approach with
respect to the arrival rate process at targets. We use the same parameters as
in case II and solve the problem for deterministic $\sigma_{i}(t)=0.5$ using
the elliptical trajectories and Fourier trajectories. The same mission is simulated assuming that ${\sigma_{i}(t)}$ is a stochastic process with
piecewise linear arrival rate. The average arrival rate is kept at $0.5$. The
results in Figs. \ref{CaseIIIFig1},\ref{CaseIIIFig2},\ref{CaseIIIFig3} and
Table \ref{CaseIIITable} show that the Fourier parametric trajectories achieve
almost the same performance by the optimization algorithm in
the stochastic setting.
\begin{figure*}
\begin{center}
\begin{subfigure}[Elliptical Trajectories for case III\label{CaseIIIFig1}]{
\includegraphics[width=1.7in,height=1.9in]{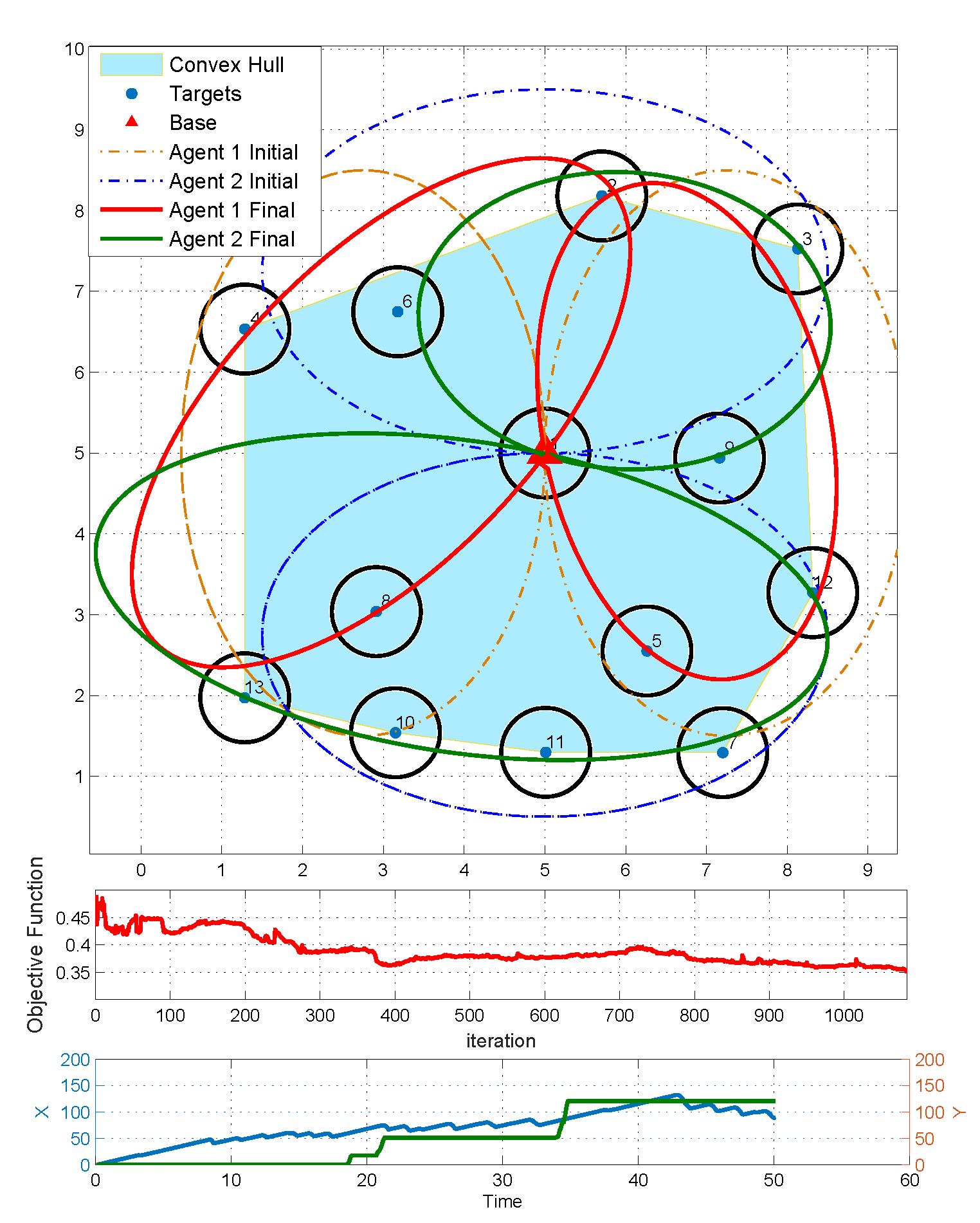}}
\end{subfigure}
\begin{subfigure}[Fourier Trajectories for case III\label{CaseIIIFig2}]{
\includegraphics[width=1.7in,height=1.9in]{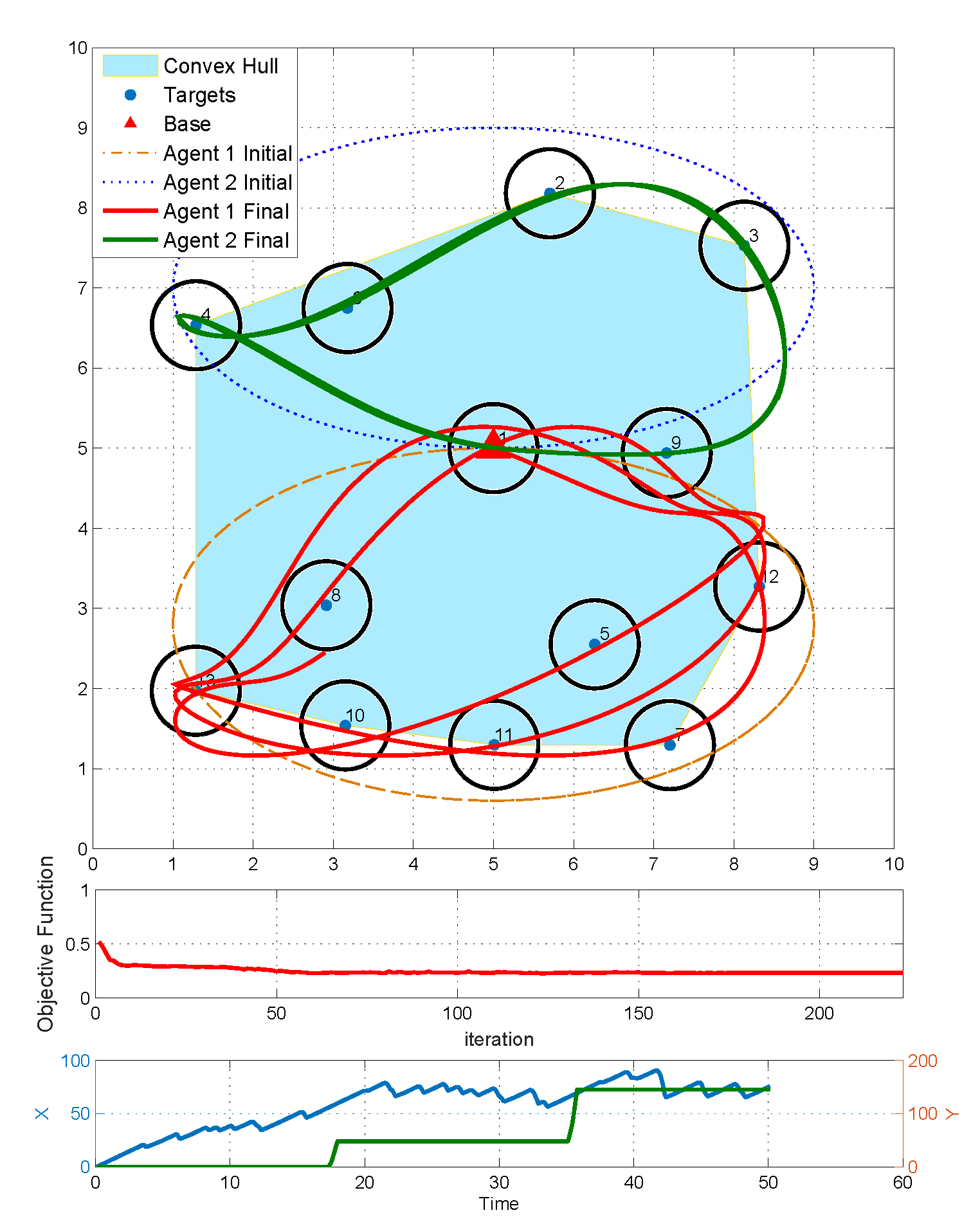}}
\end{subfigure}
\begin{subfigure}[Fourier Trajectories for case III with Stochastic Arrival\label{CaseIIIFig3}]{
\includegraphics[width=1.7in,height=1.9in]{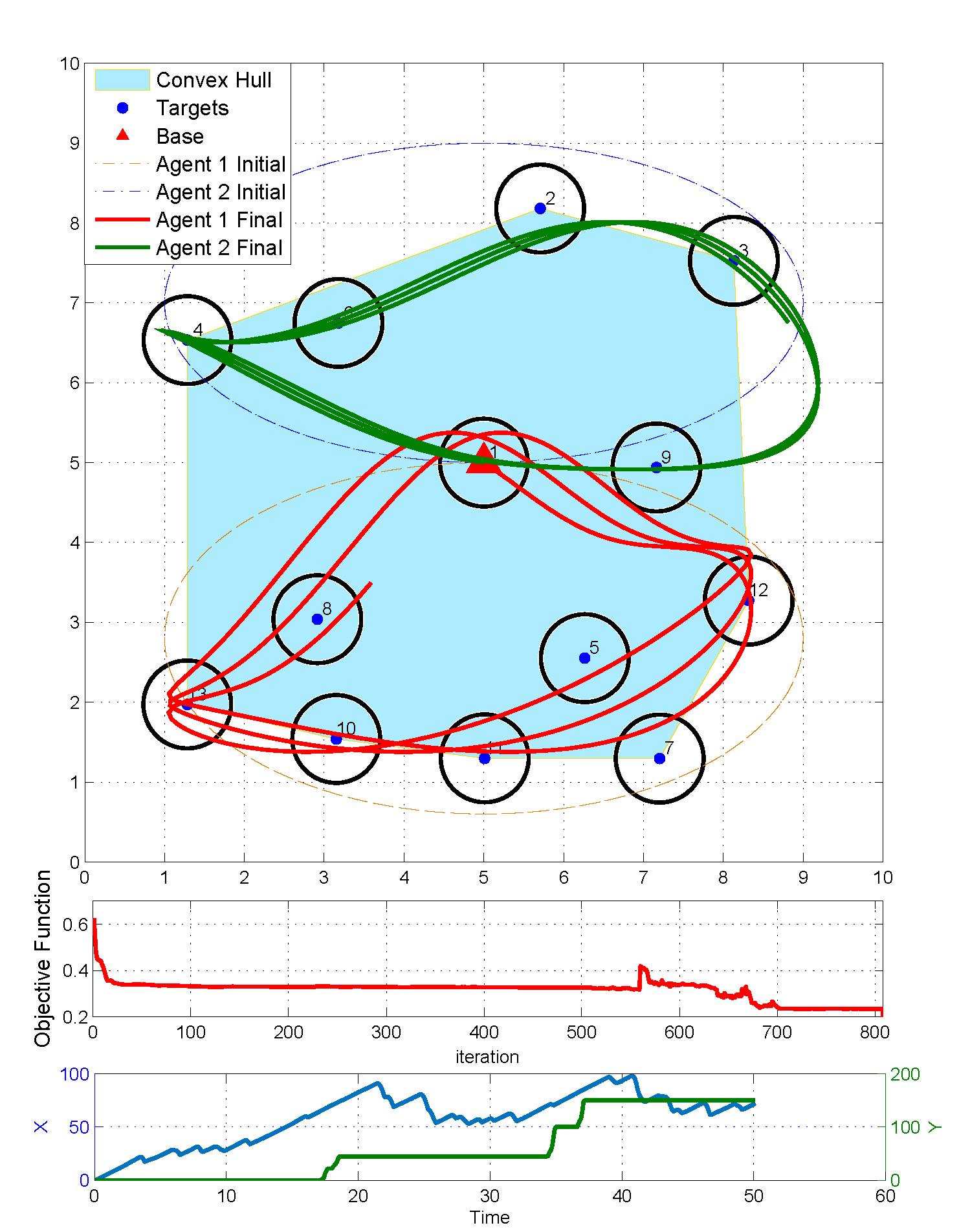}}
\end{subfigure}
\end{center}
\caption{Simulation results for the 12-target and two-agent case}
\end{figure*}
\section{Comparison with a Graph Based Algorithm}

We begin with the observation that the final parametric trajectories provide a
sequence of targets visits, similar to the functionality of a tour selection
algorithm that uses the underlying graph topology of the mission space to
determine such sequences.

We have compared the results of our approach with a graph topology algorithm
called Path Splitter Heuristic (PSH) developed in \cite{Moazzez-Estanjini2012}%
. This algorithm starts with the best Hamiltonian sequence and then uses a
heuristic method to divide the Hamiltonian tour into several sub-tours that go
through a few targets and then return to the base. The algorithm then provides
a sequence of these sub-tours for each agent. We compare the sequences from
Case I and Case II in both elliptical and Fourier trajectories and results are
shown in Tables \ref{CaseIIPSH} and \ref{CaseIIIPSH}. For a fair comparison, we
adopt each sequence and apply it with the system dynamics in our model, i.e.,
an agent can collect the data once within range of a target and the data
collection does not happen instantaneously. This, however, is not the basic
modeling assumption used in PSH, where agents pick up all the data at the
target instantaneously once at the target location. We compare the sum total
of data at targets and the base for $T=200$. A larger value of $T$ is used for
this comparison in order to approximate infinite time results. These sequences
are shown in Figs. \ref{CaseIIISeqFig1}, \ref{CaseIIISeqFig2} and
\ref{CaseIIISeqFig3} where each color represents one agent trajectory. We can
see from these comparisons that in the graph-based approach targets are
completely divided between agents. This generates a spatial partitioning,
giving each agent full responsibility for a set of targets. However, in the
trajectory planning results, in most cases, we see a temporal partitioning
where agents can visit the same targets but at different times of the mission.
This clearly allows for more robustness with respect to potential agent
failures or changes in an agent's operation. Moreover, even though the
computational complexity of the PSH algorithm and the parametric trajectory
optimization approach are comparable, methods such as PSH need to re-solve the
complete problem each time a new target may appear in the mission space. In
contrast, the on-line event-driven parametric optimization process is a
methodology designed to adapt to targets which may randomly appear in
the mission space. These results are not necessarily aiming to
prove performance enhancement but to put the two approaches
into contrast in terms of suitability for online and offline applications. Also, PSH agent trajectories consist of straight line segments. These are not physically realizable, given limitations on the motion of agents which must smoothly turn direction from one target to the next. Whereas the parametrical trajectories are easier to realize by most of the agents given these motion limitations.
\begin{figure*}
\begin{center}
\begin{subfigure}[PSH Sequences for case III\label{CaseIIISeqFig1}
]{
\includegraphics[width=1.7in]{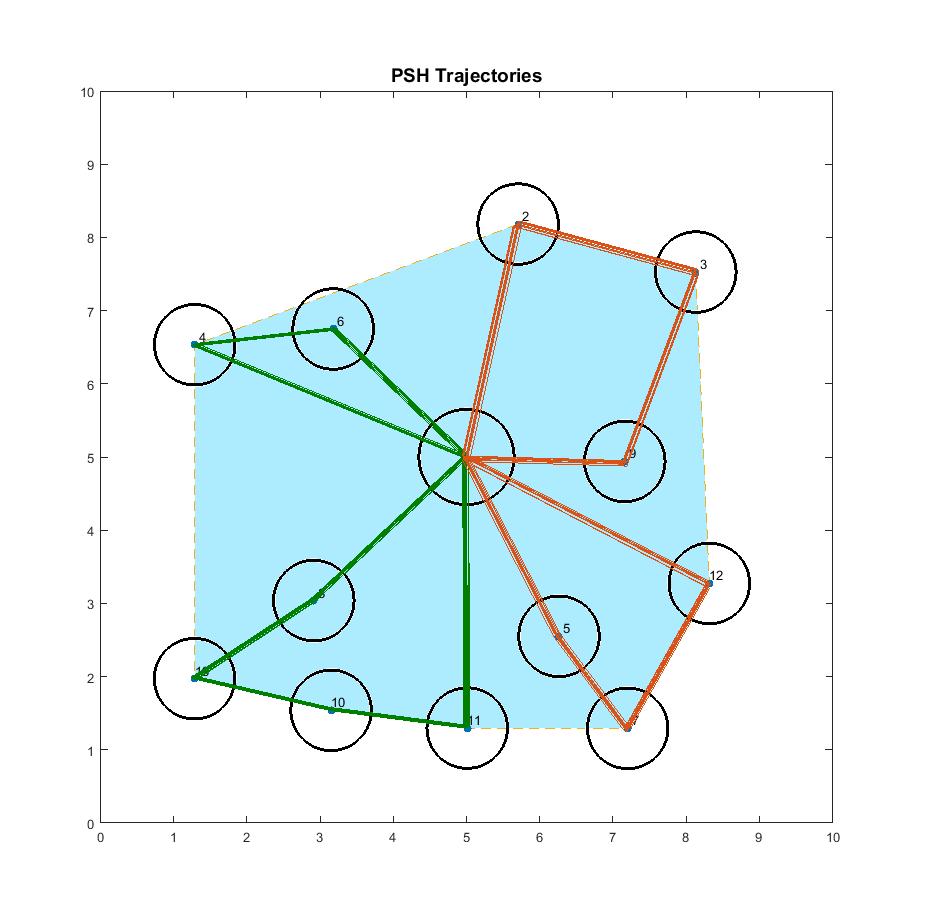}}
\end{subfigure}
\begin{subfigure}[Elliptical Sequences for case III\label{CaseIIISeqFig2}
]{
\includegraphics[width=1.75in]{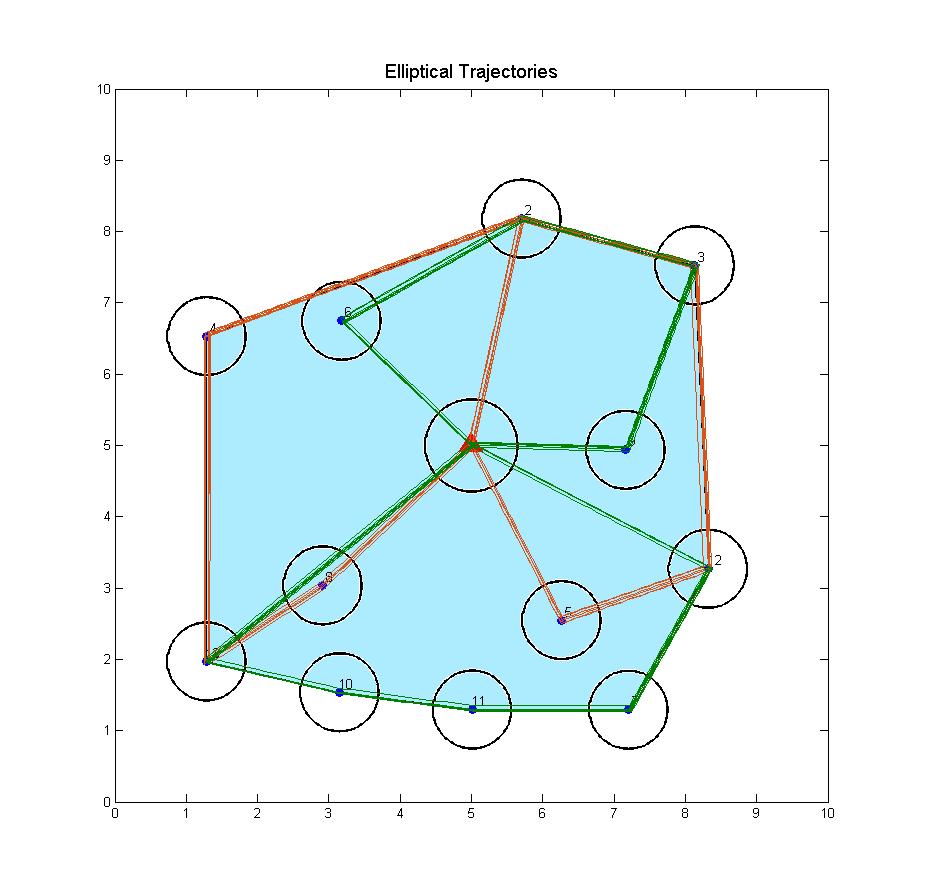}}
\end{subfigure}
\begin{subfigure}[Fourier Sequences for case III\label{CaseIIISeqFig3}
]{
\includegraphics[width=1.7in]{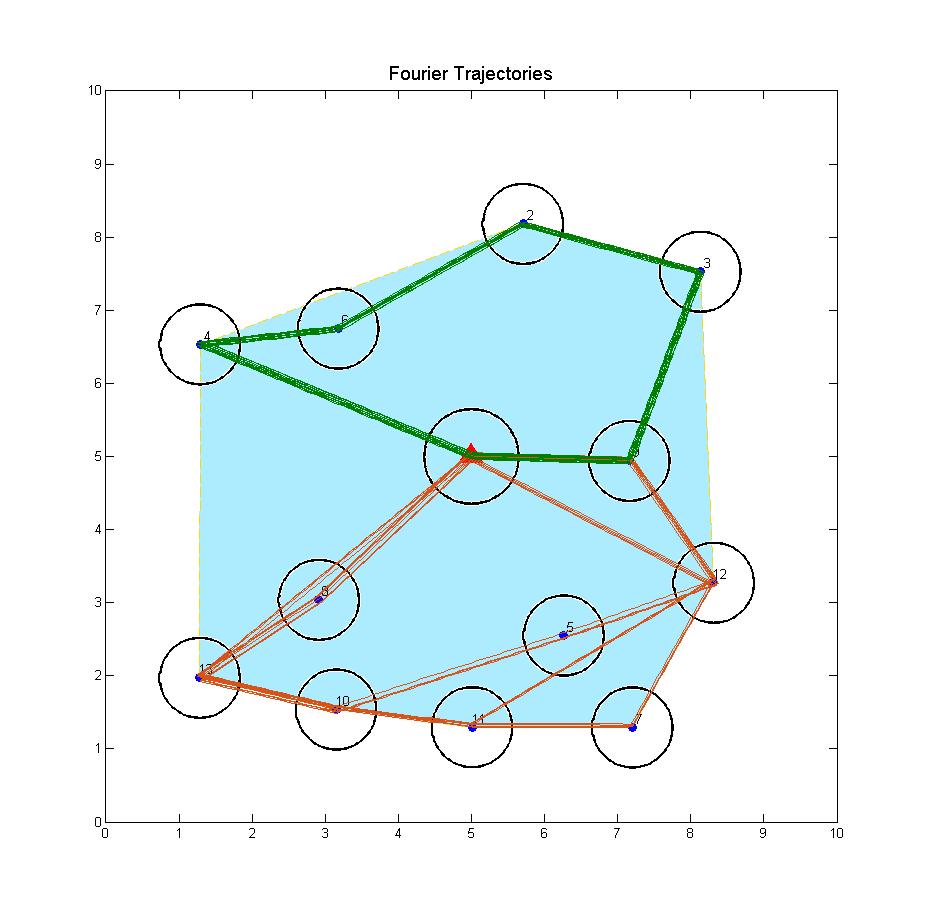}}
\end{subfigure}
\caption{Sequence Comparison}
\end{center}
\end{figure*}
\begin{table}
\caption{Result Comparison with PSH for Case II}%
\label{CaseIIPSH}
\centering
{\normalsize \
\begin{tabular}
[c]{|c|c|c|}\hline
\textbf{Method} & $J_{1}^{*}$ & $-J_{2}^{*}$\\\hline
\emph{PSH Sequence} & 0.023 & -0.22\\\hline
\emph{Elliptical Sequence} & 0.027 & -0.21\\\hline
\emph{Fourier Sequence} & 0.024 & -0.21\\\hline
\end{tabular}
}\end{table}

\begin{table}
\caption{Result Comparison with PSH for Case III}%
\label{CaseIIIPSH}
\centering
{\normalsize \
\begin{tabular}
[c]{|c|c|c|}\hline
\textbf{Method} & $J_{1}^{*}$ & $-J_{2}^{*}$\\\hline
\emph{PSH Sequence} & 0.0257 & -0.21\\\hline
\emph{Elliptical Sequence} & 0.0304 & -0.199\\\hline
\emph{Fourier Sequence} & 0.0212 & -0.21\\\hline
\end{tabular}
}\end{table}

\section{Conclusions}

\label{Conclusion} We have demonstrated a new event-driven methodology for
on-line trajectory optimization with application in the data harvesting
problem. We proposed a new performance measure that addresses the event
excitation problem in event-driven controllers. The proposed optimal control
problem is then formulated as a parametric trajectory optimization utilizing
general function families which can be subsequently optimized on line through
the use of Infinitesimal Perturbation Analysis (IPA). Several numerical
results are provided for the case of elliptical and Fourier series
trajectories and some properties of the solution are identified, including
robustness with respect to the stochastic data generation processes and
scalability in the size of the event set characterizing the underlying hybrid
dynamic system.

Although the proposed methodology is focused on applying the event-driven
optimization approach to the data harvesting problem, it should be noted that
the new metric which was introduced in \cite{Khazaeni2016} to ensure event excitation allows us to
generalize the methodology to other applications as well. The new metric
introduces a potential field or density map over the entire mission space.
This can viewed as a probability distribution of potential targets in problems
where the exact locations of targets are unknown. Used as a prior
distribution, it can be improved while the agents move within the mission
space and gather more information. In addition, this density can be
dynamically changing if the targets are moving and their location changes with
time assuming some prior information about the target paths. This provides a
tool to apply agent trajectory optimization in a much broader range of
problems tracking moving points of interest.

\begin{appendices}
\section{Proofs}
\label{appendixproofs}
Proposition \ref{Rproof}:

\begin{proof}
 Fixing $X_{i}(t)=x_{i}(t)$, from (\ref{J2R}) we have%
\begin{align*}
\int_{\mathcal{C}}R(w,t)  &  =\int_{\mathcal{C}}\sum_{i=1}^{M}\frac{\alpha
_{i}x_{i}(t)}{d_{i}^{+}(w)}dw  =\sum_{i=1}^{M}\alpha_{i}\int_{\mathcal{C}}\frac{x_{i}(t)}{d_{i}^{+}(w)}dw
\end{align*}
To evaluate $\int_{\mathcal{C}}\frac{x_{i}(t)}{d_{i}^{+}(w)}$ for each target
$i$, we first look at the case of a single target in a $2D$ space and
temporarily replace $\mathcal{C}$ by a disk with radius $\Lambda$ around the
target (black circle with radius $\Lambda$ in Fig. \ref{threetarget}). We can
now calculate the integral above using polar coordinates:
\begin{align*}
&  \int_{\mathcal{C}}\frac{x_{i}(t)}{d_{i}^{+}(w)}dw=\int_{0}^{2\pi}\int
_{0}^{\Lambda}\frac{x_{i}(t)}{\max(r_{i},r)}drd\theta\\
&  =\int_{0}^{2\pi}\int_{0}^{r_{i}}\frac{x_{i}(t)}{r_{i}}drd\theta+\int
_{0}^{2\pi}\int_{r_{i}}^{\Lambda}\frac{x_{i}(t)}{r}drd\theta\\
&  =x_{i}(t)\big[2\pi\big(1+\log(\frac{\Lambda}{r_{i}})\big)\big]
\end{align*}
\begin{figure}[ptb]
\begin{center}
%Requires \usepackage{graphicx}
\includegraphics[width=1.7in]{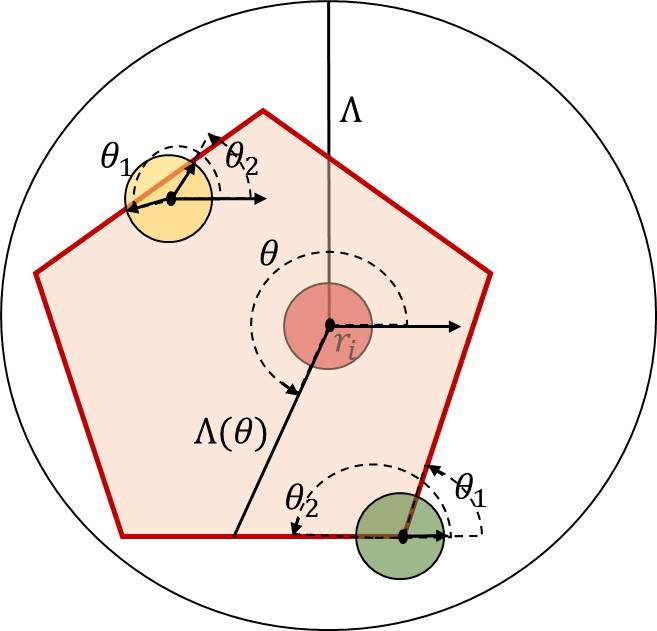}\newline
\end{center}
\caption{One Target $R(w,t)$ Calculation}%
\label{threetarget}%
\end{figure}Since $\mathcal{C}$ is actually the convex hull of all targets, we
will use the same idea to calculate $\int_{\mathcal{C}}\frac{x_{i}(t)}%
{d_{i}^{+}(w)}dw$ where $\mathcal{C}$ is the convex hull by considering three
separate cases depending on the target location.

\emph{1. Target $i$ and $C(w_{i})$ are completely in the interior of
$\mathcal{C}$:} This is shown in Fig. \ref{threetarget} for the red target.
Using the same polar coordinate for each $\theta$, we define $\Lambda(\theta)$
to be the distance of the target to the edge of $\mathcal{C}$ in the direction
of $\theta$ and get:
\begin{align}
&  \int_{\mathcal{C}}\frac{x_{i}(t)}{d_{i}^{+}(w)}dw=\int_{0}^{2\pi}\int
_{0}^{\Lambda(\theta)}\frac{x_{i}(t)}{d_{i}^{+}(r,\theta)}drd\theta\nonumber\\
&  =\int_{0}^{2\pi}\int_{0}^{r_{i}}\frac{x_{i}(t)}{r_{i}}drd\theta+\int
_{0}^{2\pi}\int_{r_{i}}^{\Lambda(\theta)}\frac{x_{i}(t)}{r}drd\theta
\nonumber\\
&  =x_{i}(t)\big[2\pi+\int_{0}^{2\pi}\log(\frac{\Lambda(\theta)}{r_{i}%
})d\theta\big] \label{Rcalcintegral}%
\end{align}
The second part in \eqref{Rcalcintegral} is calculated knowing that
$\Lambda(\theta)\geq r_{i}$. This means $\log(\frac{\Lambda(\theta)}{r_{i}%
})>0$ and the term in brackets is positive. We can then define $c_{i}$ in
\eqref{corequation} as
\begin{equation}
c_{i}=\alpha_{i}\big[2\pi+\int_{0}^{2\pi}\log(\frac{\Lambda(\theta)}{r_{i}%
})d\theta\big]
\end{equation}
\emph{2. Target $i$ is on an edge of $\mathcal{C}$:} This is shown in Fig.
\ref{threetarget} for the green target. Since $C(w_{i})$ is not entirely
contained within $\mathcal{C}$, we have
\begin{align}
&  \int_{\mathcal{C}}\frac{x_{i}(t)}{d_{i}^{+}(w)}dw=\int_{0}^{2\pi}\int
_{0}^{\Lambda(\theta)}\frac{x_{i}(t)}{d_{i}^{+}(r,\theta)}drd\theta\nonumber\\
&  =\int_{\theta_{1}}^{\theta_{2}}\int_{0}^{r_{i}}\frac{x_{i}(t)}{r_{i}%
}drd\theta+\int_{\theta_{1}}^{\theta_{2}}\int_{r_{i}}^{\Lambda(\theta)}%
\frac{x_{i}(t)}{r}drd\theta\nonumber\\
&  =x_{i}(t)\big[\theta_{2}-\theta_{1}+\int_{\theta_{1}}^{\theta_{2}}%
\log(\frac{\Lambda(\theta)}{r_{i}})d\theta\big] \label{Rcalcintegral1}%
\end{align}
The second part in \eqref{Rcalcintegral1} is calculated knowing that for
$\theta\in\lbrack\theta_{1}~\theta_{2}]$ we have $\Lambda(\theta)\geq r_{i}$.
Therefore, $\log(\frac{\Lambda(\theta)}{r_{i}})>0$ and by definition
$\theta_{2}-\theta_{1}>0$ so that the term in brackets is positive. We can
then define $c_{i}$ in \eqref{corequation} as
\begin{equation}
c_{i}=\alpha_{i}\big[\theta_{2}-\theta_{1}+\int_{\theta_{1}}^{\theta_{2}}%
\log(\frac{\Lambda(\theta)}{r_{i}})d\theta\big]
\end{equation}
\emph{3. Target $i$ is in the interior of $\mathcal{C}$ but $C(w_{i})$ is not
completely in the interior of $\mathcal{C}$:} This is the case of the yellow
target in Fig. \ref{threetarget}. Carrying out the integration for the
appropriate limits, since $C(w_{i})$ is not fully contained in $\mathcal{C}$:

\begin{align}
&  \int_{\mathcal{C}}\frac{x_{i}(t)}{d_{i}^{+}(w)}dw=\int_{0}^{2\pi}\int
_{0}^{\Lambda(\theta)}\frac{x_{i}(t)}{d_{i}^{+}(r,\theta)}drd\theta\nonumber\\
&  =\int_{0}^{2\pi}\int_{0}^{r(\theta)}\frac{x_{i}(t)}{r_{i}}drd\theta
+\int_{\theta_{1}}^{\theta_{2}}\int_{r_{i}}^{\Lambda(\theta)}\frac{x_{i}%
(t)}{r}drd\theta\nonumber\\
&  =x_{i}(t)\big[\int_{0}^{2\pi}\frac{r(\theta)}{r_{i}}d\theta+\int
_{\theta_{1}}^{\theta_{2}}\log(\frac{\Lambda(\theta)}{r_{i}})d\theta
\big] \label{Rcalcintegral2}%
\end{align}
Here, $\Lambda(\theta)$ is only defined for $\theta\in\lbrack\theta_{1}%
~\theta_{2}]$ as shown in Fig. \ref{threetarget} for the yellow target. Again,
since $\Lambda(\theta)\geq r_{i}$, it follows that $\log(\frac{\Lambda
(\theta)}{r_{i}})>0$ and with $r(\theta)>0$ the term in brackets above is
positive. We can define $c_{i}$ in \eqref{corequation} as
\begin{equation}
c_{i}=\alpha_{i}\big[\int_{0}^{2\pi}\frac{r(\theta)}{r_{i}}d\theta
+\int_{\theta_{1}}^{\theta_{2}}\log(\frac{\Lambda(\theta)}{r_{i}})d\theta\big]
\end{equation}
\end{proof}
\section{Elliptical Trajectories}
\label{AppEllipse} In order to calculate the IPA derivatives we need the
derivatives of state variable associated with agents with respect to all
entries in the parameter vector $\Theta_{j}=[A_{j},B_{j},a_{j},b_{j},\phi
_{j}]$ for all agents $j$. These derivatives do not depend on the events in
the system, since the agent trajectories are fixed at each iteration. For now
we assume $\mathcal{E}_{j}=1$ for all $j=1,\dots,N$ hence, we drop the
superscript. We have:
\begin{equation}
\frac{\partial s_{j}^{x}}{\partial A_{j}}=1,\qquad\frac{\partial s_{j}^{x}%
}{\partial B_{j}}=0\label{dsxdAB}%
\end{equation}%
\begin{equation}
\frac{\partial s_{j}^{x}}{\partial a_{j}}=\cos\rho_{j}(t)\cos\phi_{j}%
,\qquad\frac{\partial s_{j}^{x}}{\partial b_{j}}=-\sin\rho_{j}(t)\sin\phi
_{j}\label{dsxdab}%
\end{equation}%
\begin{equation}
\frac{\partial s_{j}^{x}}{\partial\phi_{j}}=-a_{j}\cos\rho_{j}(t)\sin\phi
_{j}-b_{j}\sin\rho_{j}(t)\cos\phi_{j}\label{dsxdphi}%
\end{equation}%
\begin{equation}
\frac{\partial s_{j}^{y}}{\partial A_{j}}=0,\qquad\frac{\partial s_{j}^{y}%
}{\partial B_{j}}=1\label{dsydAB}%
\end{equation}%
\begin{equation}
\frac{\partial s_{j}^{y}}{\partial a_{j}}=\cos\rho_{j}(t)\sin\phi_{j}%
,\qquad\frac{\partial s_{j}^{y}}{\partial b_{j}}=\sin\rho_{j}(t)\cos\phi
_{j}\label{dsydab}%
\end{equation}%
\begin{equation}
\frac{\partial s_{j}^{y}}{\partial\phi_{j}}=a_{j}\cos\rho_{j}(t)\cos\phi
_{j}-b_{j}\sin\rho_{j}(t)\sin\phi_{j}\label{dsydphi}%
\end{equation}
The time derivative of the position state variables are calculated as
follows:
\begin{equation}
\resizebox{0.89 \columnwidth}{!}{$
\dot{s}_{j}^{x}(t)=-a_{j}\dot{\rho}_{j}(t)\sin\rho_{j}(t)\cos\phi_{j}%
+b_{j}\dot{\rho}_{j}(t)\cos\rho_{j}(t)\sin\phi_{j}$}\label{dsxdt}%
\end{equation}%
\begin{equation}
\resizebox{0.89 \columnwidth}{!}{$
\dot{s}_{j}^{y}(t)=-a_{j}\dot{\rho}_{j}(t)\sin\rho_{j}(t)\sin\phi_{j}%
+b_{j}\dot{\rho}_{j}(t)\cos\rho_{j}(t)\cos\phi_{j}$}\label{dsydt}%
\end{equation}
The gradient of the last term in the $J_{e}$ in \eqref{ParamOptim_e} needs to
be calculated separately. We have for $j\neq l$, $\frac{\partial
\mathcal{C}_{j}}{\partial\Theta_{l}}=0$ and for $j=l$:
%\begin{equation}
%\nabla\mathcal{C}_j(\Theta_j)=[\frac{\partial \mathcal{C}_j}{\partial A_1} ~ \frac{\partial \mathcal{C}_j}{\partial B_1} ~ \frac{\partial \mathcal{C}_j}{\partial a_1} ~ \frac{\partial \mathcal{C}_j}{\partial b_1} ~ \frac{\partial \mathcal{C}_j}{\partial \phi_1} \dots \frac{\partial \mathcal{C}_j}{\partial A_N} ~ \frac{\partial \mathcal{C}_j}{\partial B_N} ~ \frac{\partial \mathcal{C}_j}{\partial a_N} ~ \frac{\partial \mathcal{C}_j}{\partial b_N} ~ \frac{\partial \mathcal{C}_j}{\partial \phi_N}]
%\end{equation}%
\[
\frac{\partial\mathcal{C}_{j}}{\partial A_{j}}=2\mathcal{C}_{j}\big(-\cos
^{2}\phi_{j}\frac{\partial f_{j}^{1}}{\partial A_{j}}-\sin^{2}\phi_{j}%
\frac{\partial f_{j}^{2}}{\partial A_{j}}-\sin2\phi_{j}\frac{\partial
f_{j}^{3}}{\partial A_{j}}\big)
\]%
\[
\frac{\partial\mathcal{C}_{j}}{\partial B_{j}}=2\mathcal{C}_{j}\big(-\cos
^{2}\phi_{j}\frac{\partial f_{j}^{1}}{\partial B_{j}}-\sin^{2}\phi_{j}%
\frac{\partial f_{j}^{2}}{\partial B_{j}}-\sin2\phi_{j}\frac{\partial
f_{j}^{3}}{\partial B_{j}}\big)
\]%
\[
\frac{\partial\mathcal{C}_{j}}{\partial a_{j}}=2\mathcal{C}_{j}\big(-\cos
^{2}\phi_{j}\frac{\partial f_{j}^{1}}{\partial a_{j}}-\sin^{2}\phi_{j}%
\frac{\partial f_{j}^{2}}{\partial a_{j}}-\sin2\phi_{j}\frac{\partial
f_{j}^{3}}{\partial a_{j}}\big)
\]%
\[
\frac{\partial\mathcal{C}_{j}}{\partial b_{j}}=2\mathcal{C}_{j}\big(-\cos
^{2}\phi_{j}\frac{\partial f_{j}^{1}}{\partial b_{j}}-\sin^{2}\phi_{j}%
\frac{\partial f_{j}^{2}}{\partial b_{j}}-\sin2\phi_{j}\frac{\partial
f_{j}^{3}}{\partial b_{j}}\big)
\]%
\[
\frac{\partial\mathcal{C}_{j}}{\partial\phi_{j}}=2\mathcal{C}_{j}%
\big((f_{j}^{1}-f_{j}^{2})\sin2\phi_{j}-2f_{j}^{3}\cos2\phi_{j}\big)
\]
where
\[%
\begin{split}
\frac{\partial f_{j}^{1}}{\partial A_{j}} &  =-2\Big(\frac{w_{\!_{B}}%
^{x}-A_{j}}{a_{j}^{2}}\Big),\quad\frac{\partial f_{j}^{1}}{\partial B_{j}%
}=-2\Big(\frac{w_{\!_{B}}^{y}-B_{j}}{b_{j}^{2}}\Big)\\
\frac{\partial f_{j}^{1}}{\partial a_{j}} &  =-2\Big(\frac{(w_{\!_{B}}%
^{x}-A_{j})^{2}}{a_{j}^{3}}\Big),\quad\frac{\partial f_{j}^{1}}{\partial
b_{j}}=-2\Big(\frac{(w_{\!_{B}}^{y}-B_{j})^{2}}{b_{j}^{3}}\Big)\\
&
\end{split}
\]%
\[%
\begin{split}
\frac{\partial f_{j}^{2}}{\partial A_{j}} &  =-2\Big(\frac{w_{\!_{B}}%
^{x}-A_{j}}{b_{j}^{2}}\Big),~\frac{\partial f_{j}^{2}}{\partial B_{j}%
}=-2\Big(\frac{w_{\!_{B}}^{y}-B_{j}}{a_{j}^{2}}\Big)\\
\frac{\partial f_{j}^{2}}{\partial a_{j}} &  =-2\Big(\frac{(w_{\!_{B}}%
^{y}-B_{j})^{2}}{a_{j}^{3}}\Big),~\frac{\partial f_{j}^{2}}{\partial b_{j}%
}=-2\Big(\frac{(w_{\!_{B}}^{x}-A_{j})^{2}}{b_{j}^{3}}\Big)\\
&
\end{split}
\]%
\[%
\begin{split}
\frac{\partial f_{j}^{3}}{\partial A_{j}} &  =-\Big(\frac{(b_{j}^{2}-a_{j}%
^{2})(w_{\!_{B}}^{y}-B_{j})}{a_{j}^{2}b_{j}^{2}}\Big)\\
\frac{\partial f_{j}^{3}}{\partial B_{j}} &  =-\Big(\frac{(b_{j}^{2}-a_{j}%
^{2})(w_{\!_{B}}^{x}-A_{j})}{a_{j}^{2}b_{j}^{2}}\Big)
\end{split}
\]%
\[%
\begin{split}
\frac{\partial f_{j}^{3}}{\partial a_{j}} &  =-2\Big(\frac{(w_{\!_{B}}%
^{x}-A_{j})(w_{\!_{B}}^{y}-B_{j})}{a_{j}^{3}}\Big)\\
\frac{\partial f_{j}^{3}}{\partial b_{j}} &  =2\Big(\frac{(w_{\!_{B}}%
^{x}-A_{j})(w_{\!_{B}}^{y}-B_{j})}{b_{j}^{3}}\Big)
\end{split}
\]
\section{Fourier Series Trajectories}
In the Fourier parametric trajectories the agent state derivative is
calculated as follows. The parameter vector is $\Theta_{j}=[f_{j}^{x}%
,a_{0,j},\ldots,a_{\Gamma_{j}^{x}},b_{0,j},\ldots,b_{\Gamma_{j}^{y}}%
,\phi_{1,j},\ldots,\phi_{\Gamma_{j}^{x}},\xi_{1,j},\ldots,\xi_{\Gamma_{j}^{y}%
}]$. Thus, we have:
\begin{equation}
\frac{\partial s_{j}^{x}}{\partial a_{0,j}}=1,\qquad\frac{\partial s_{j}^{x}%
}{\partial b_{0,j}}=0
\end{equation}%
\begin{equation}
\frac{\partial s_{j}^{x}}{\partial a_{n,j}}=\sin(2\pi nf_{j}^{x}\rho
_{j}(t)+\phi_{n,j}^{x}),\qquad\frac{\partial s_{j}^{x}}{\partial b_{n,j}}=0
\end{equation}%
\begin{equation}
\frac{\partial s_{j}^{x}}{\partial\phi_{n,j}^{x}}=a_{n,j}\cos(2\pi nf_{j}%
^{x}\rho_{j}(t)+\phi_{n,j}^{x})\qquad\frac{\partial s_{j}^{x}}{\partial
\phi_{n,j}^{y}}=0
\end{equation}%
\begin{equation}
\frac{\partial s_{j}^{x}}{\partial f_{j}^{x}}=2\pi\rho_{j}(t)\displaystyle\sum
_{n=1}^{\Gamma_{j}^{x}}a_{n,j}n\cos(2\pi nf_{j}^{x}\rho_{j}(t)+\phi_{n,j}%
^{x}),
\end{equation}%
\begin{equation}
\frac{\partial s_{j}^{y}}{\partial b_{0,j}}=1,\qquad\frac{\partial s_{j}^{y}%
}{\partial a_{0,j}}=0
\end{equation}%
\begin{equation}
\frac{\partial s_{j}^{y}}{\partial b_{n,j}}=\sin(2\pi nf_{j}^{y}\rho
_{j}(t)+\phi_{n,j}^{y}),\qquad\frac{\partial s_{j}^{y}}{\partial a_{n,j}}=0
\end{equation}%
\begin{equation}
\frac{\partial s_{j}^{y}}{\partial\phi_{n,j}^{y}}=b_{n,j}\cos(2\pi nf_{j}%
^{y}\rho_{j}(t)+\phi_{n,j}^{y})\qquad\frac{\partial s_{j}^{y}}{\partial
\phi_{n,j}^{x}}=0
\end{equation}%
\begin{equation}
\frac{\partial s_{j}^{y}}{\partial f_{j}^{x}}=0
\end{equation}
The time derivative of the position state variables are calculated as
follows:
\begin{equation}
\dot{s}_{j}^{x}(t)=\dot{\rho}_{j}(t)\sum_{n=1}^{\Gamma_{j}^{x}}2\pi nf_{j}%
^{x}a_{n,j}\cos(2\pi nf_{j}^{x}\rho_{j}(t)+\phi_{n,j}^{x}),
\end{equation}%
\begin{equation}
\dot{s}_{j}^{y}(t)=\dot{\rho}_{j}(t)\sum_{n=1}^{\Gamma_{j}^{y}}2\pi nf_{j}%
^{y}b_{n,j}\cos(2\pi nf_{j}^{y}\rho_{j}(t)+\phi_{n,j}^{y}),
\end{equation}
\section{IPA events and derivatives}
\label{IPAapp} In this section, we derive all event time derivatives and state
derivatives with respect to the controllable parameter $\Theta$ for each event
by applying the IPA equations.\newline1. \textbf{Event $\xi_{i}^{0}$}: This
event causes a transition from $X_{i}(t)>0$, $t<\tau_{k}$ to $X_{i}(t)=0$,
$t\geq\tau_{k}$. The switching function is $g_{k}(\Theta,\mathbf{X})=X_{i}$ so
$\frac{\partial g_{k}}{\partial X_{i}}=1$. From (\ref{tauprime}) and
\eqref{Xdot}:
\begin{equation}%
\begin{split}
{\tau_{k}^{\prime}}  &  =-\Big(\frac{\partial g_{k}}{\partial X_{i}}f_{k}%
(\tau_{k}^{-})\Big)^{-1}\Big(g^{\prime}_{k}+\frac{\partial g_{k}}{\partial
X_{i}}{X_{i}^{\prime}(\tau_{k}^{-})}\Big)\\
&  =-\frac{X_{i}^{\prime}(\tau_{k}^{-})}{\sigma_{i}(\tau_{k})-\mu_{ij}%
p_{ij}(\tau_{k})}%
\end{split}
\label{tauprime1}%
\end{equation}
where agent $j$ is the one connected to $i$ at $t=\tau_{k}$ and we have used
the assumption that two events occur at the same time w.p. $0$, hence
$\sigma_{i}(\tau_{k}^{-})=\sigma_{i}(\tau_{k})$. From (\ref{dXdtprime}%
)-(\ref{Xprime}), since $\dot{X}_{i}(t)=0$, for $\tau_{k}\leq t<\tau_{k+1}$:
\begin{equation}
\frac{d}{dt}X_{i}^{\prime}(t)=\frac{\partial\dot{X}_{i}(t)}{\partial X_{i}%
(t)}X_{i}^{\prime}(t)+{\dot{X}_{i}^{\prime}(t)}=0 \label{Xprimet1}%
\end{equation}%
\begin{equation}%
\begin{split}
&  X_{i}^{\prime}(\tau_{k}^{+})=X_{i}^{\prime}(\tau_{k}^{-})+\Big[\Big(\sigma
_{i}(\tau_{k})-\mu_{ij}p_{ij}(\tau_{k})\Big)-0\Big]{\tau_{k}}^{\prime}\\
&  =X_{i}^{\prime}(\tau_{k}^{-})-\frac{X_{i}^{\prime}(\tau_{k}^{-}%
)\Big(\sigma_{i}(\tau_{k})-\mu_{ij}p_{ij}(\tau_{k})\Big)}{\sigma_{i}(\tau
_{k})-\mu_{ij}p_{ij}(\tau_{k})}=0
\end{split}
\label{Xprime0}%
\end{equation}
For $X_{r}(t)$, $r\neq i$, the dynamics of $X_{r}(t)$ in \eqref{Xdot} are
unaffected and we have:
\begin{equation}
X_{r}^{\prime}(\tau_{k}^{+})=X_{r}^{\prime}(\tau_{k}^{-}) \label{Xrprime0}%
\end{equation}
If $X_{r}(\tau_{k})>0$ and agent $l$ is connected to it, then%
\begin{equation}%
\begin{split}
&  \frac{d}{dt}X_{r}^{\prime}(t)=\frac{\partial\dot{X}_{r}(t)}{\partial
X_{r}(t)}X_{r}^{\prime}(t)+{\dot{X}_{r}^{\prime}(t)}\\
&  =\sigma_{r}^{\prime}(t)-\mu_{rl}p^{\prime}_{rl}(\tau_{k})=-\mu_{rl}%
p_{rl}^{\prime}(t)
\end{split}
\label{dxdt}%
\end{equation}
and if ${X}_{r}(t)=0$ in $[\tau_{k},\tau_{k+1}]$ or if no agents are connected
to $i$, then $\frac{d}{dt}X_{r}^{\prime}(t)=0$.\newline For $Y_{r}(t)$,
$r=1,\dots,M$, the dynamics of $Y_{r}(t)$ in \eqref{Ydot} are not affected by
the event $\xi_{i}^{0}$ at $\tau_{k}$, hence
\begin{equation}
Y_{r}^{\prime}(\tau_{k}^{+})=Y_{r}^{\prime}(\tau_{k}^{-}) \label{yrplus}%
\end{equation}
and since $\dot{Y}_{r}(t)=\beta_{r}(t)$, for $\tau_{k}\leq t<\tau_{k+1}$:
\begin{equation}
\frac{d}{dt}Y_{r}^{\prime}(t)=\frac{\partial\dot{Y}_{r}(t)}{\partial Y_{r}%
(t)}Y_{r}^{\prime}(t)+{\dot{Y}_{r}^{\prime}(t)}=\beta_{r}^{\prime}(t)
\label{dydt}%
\end{equation}
For $Z_{ij}(t)$, we must have $Z_{ij}(\tau_{k})>0$ since $X_{i}(\tau_{k}%
^{-})>0$, hence $\tilde{\mu}_{ij}(\tau_{k}^{-})>0$ and from (\ref{Xprime}):
\begin{equation}%
\begin{split}
&  Z_{ij}^{\prime}(\tau_{k}^{+})=Z_{ij}^{\prime}(\tau_{k}^{-})+\Big[\dot
{Z}_{ij}(\tau_{k}^{-})-\dot{Z}_{ij}(\tau_{k}^{+})\Big]\tau_{k}^{\prime}\\
&  =Z_{ij}^{\prime}(\tau_{k}^{-})+\Big[\tilde{\mu}_{ij}(\tau_{k}^{-}%
)-\tilde{\mu}_{ij}(\tau_{k}^{+})\Big]p_{ij}(\tau_{k})\tau_{k}^{\prime}%
\end{split}
\label{Zprime}%
\end{equation}
Since $X_{i}(\tau_{k}^{-})>0$, from \eqref{mutilde} we have $\tilde{\mu}%
_{ij}(\tau_{k}^{-})=\mu_{ij}$. At $\tau_{k}^{+}$, $j$ remains connected to
target $i$ with $\tilde{\mu}_{ij}(\tau_{k}^{+})=\sigma_{i}(\tau_{k}%
^{+})/p_{ij}(\tau_{k})=\sigma_{i}(\tau_{k})/p_{ij}(\tau_{k})$ and we get%
\begin{equation}
\resizebox{0.89 \columnwidth}{!}{$
\begin{split}
\label{Zprime01}Z_{ij}^{\prime}(\tau_{k}^{+})  &  =Z_{ij}^{\prime}(\tau
_{k}^{-})+\frac{-X_{i}^{\prime}(\tau_{k}^{-})\Big[\mu_{ij}p_{ij}(\tau
_{k})-\sigma_{i}(\tau_{k})\Big]}{\sigma_{i}(\tau_{k})-\mu_{ij}p_{ij}(\tau
_{k})}\\
&  =Z_{ij}^{\prime}(\tau_{k}^{-})+X_{i}^{\prime}(\tau_{k}^{-})
\end{split}$}
\end{equation}
From (\ref{dXdtprime}) for $\tau_{k}\leq t<\tau_{k+1}$:%
\begin{equation}%
\begin{split}
&  \frac{d}{dt}Z_{ij}^{\prime}(t)=\frac{\partial\dot{Z}_{ij}(t)}{\partial
Z_{ij}(t)}Z_{ij}^{\prime}(t)+\dot{Z}^{\prime}_{ij}(t)\\
&  =\dot{Z}^{\prime}_{ij}(t)=\Big(\tilde{\mu}_{ij}(t)P^{\prime}_{ij}%
(t)-\beta_{ij}P^{\prime}_{\!_{Bj}}(t)\Big)
\end{split}
\label{dzdt}%
\end{equation}
Since $\tilde{\mu}_{ij}(t)=\sigma_{i}(t)/p_{ij}(t)$ for the agent which
remains connected to target $i$ after this event, it follows that
$\frac{\partial}{\partial\Theta}[\tilde{\mu}_{ij}(t)p_{ij}(t)]=0$. Moreover,
$p_{\!_{Bj}}(t)=0$ by our assumption that agents cannot be within range of the
base and targets at the same time and we get%
\begin{equation}
\frac{d}{dt}Z_{ij}^{\prime}(t)=0 \label{dzdt1}%
\end{equation}
Otherwise, for $r\neq j$, we have $\tilde{\mu}_{ir}(t)=0$ and we get:
\begin{equation}
\frac{d}{dt}Z_{ir}^{\prime}(t)=-\beta_{ir}p_{\!_{Br}}^{\prime}(t)
\end{equation}
Finally, for $Z_{rj}(t)$, $r\neq i$ we have $Z_{rj}^{\prime}(\tau_{k}%
^{+})=Z_{rj}^{\prime}(\tau_{k}^{-})$. If $Z_{rj}(t)=0$ in $[\tau_{k}%
,\tau_{k+1})$, then $\frac{d}{dt}Z_{rj}^{\prime}(t)=0$. Otherwise, we get
$\frac{d}{dt}Z_{rj}^{\prime}(t)$ from \eqref{dzdt} with $i$ replaced by
$r$.\newline2. \textbf{Event $\xi_{i}^{+}$}: This event causes a transition
from $X_{i}(t)=0$, $t\le\tau_{k}$ to $X_{i}(t)>0$, $t > \tau_{k}$. Note that
this transition can occur as an exogenous event when an empty queue $X_{i}(t)$
gets a new arrival in which case we simply have $\tau_{k}^{\prime}=0$ since
the exogenous event is independent of the controllable parameters. In the
endogenous case, however, we have the switching function $g_{k}(\Theta
,\mathbf{X})=\sigma_{i}(t)- \mu_{ij}p_{ij}(t)$ in which agent $j$ is connected
to target $i$ at $t=\tau_{k}$. Assuming $s_{j}^{\prime}(t)=\big[\frac{\partial
s_{j}^{x}}{\partial\Theta}~\frac{\partial s_{j}^{y}}{\partial\Theta
}\big]^{\top}$ and $\dot s_{j}=[\dot s_{j}^{x}~ \dot s_{j}^{y}]^{\top}$, from
(\ref{tauprime}):
\begin{equation}%
\begin{split}
{\tau_{k}}^{\prime}  &  =-\Big(\frac{\partial g_{k}}{\partial s_{j}}%
s_{j}^{\prime}(\tau_{k})\Big)\Big({g_{k}^{\prime}\dot s_{j}(\tau_{k}%
)}\Big)^{-1}%
\end{split}
\end{equation}
At $\tau_{k}$ we have $\sigma_{i}(\tau_{k})=\mu_{ij}p_{ij}(\tau_{k})$.
Therefore from \eqref{Xprime}:
\begin{equation}%
\begin{split}
&  X_{i}^{\prime}(\tau_{k}^{+})=X_{i}^{\prime}(\tau_{k}^{-})+[\dot X_{i}%
(\tau_{k}^{-})-\dot X_{i}(\tau_{k}^{+})]{\tau_{k}}^{\prime}\\
&  =X_{i}^{\prime}(\tau_{k}^{-})+\Big(0-\sigma_{i}(\tau_{k})+\mu_{ij}%
p_{ij}(\tau_{k})\Big){\tau_{k}}^{\prime}=X_{i}^{\prime}(\tau_{k}^{-})
\end{split}
\label{xiplus}%
\end{equation}
Having $X_{i}(t)>0$ in $[\tau_{k},\tau_{k+1})$ we know $\dot X_{i}%
(t)=\sigma_{i}(t)-\mu_{ij}p_{ij}(t)$ therefor, we can get $\frac{d}{dt}
X_{i}^{\prime}(t)$ from \eqref{dxdt} with $r$ and $l$ replaced by $i$ and $j$.
For $X_{r}(t)$, $r \ne i$, if $X_{r}(\tau_{k})>0$ and agent $l$ is connected
to $r$ then $\dot X_{r}(\tau_{k})=\sigma_{r}(\tau_{k})-\mu_{rl}p_{rl}(\tau
_{k})$, therefor, we get $X^{\prime}_{r}(\tau_{k}^{+})$ from \eqref{Xrprime0}
while in $[\tau_{k},\tau_{k+1})$ we have $\frac{d}{dt}X^{\prime}_{r}(t)$ from
\eqref{dxdt}. If $X_{r}(\tau_{k})=0$ or if no agent is connected to target
$r$, $\dot X_{r}(\tau_{k})=0$. Thus, $X^{\prime}_{r}(\tau_{k}^{+})=X^{\prime
}_{r}(\tau_{k}^{-})$ and $\frac{d}{dt}X^{\prime}_{r}(t)=0$.\newline For
$Y_{r}(t)$, $r=1,\dots,M$ the dynamics of $Y_{r}(t)$ in \eqref{Ydot} are not
affected by the event at $\tau_{k}$ hence, we can get $Y^{\prime}_{r}(\tau
_{k}^{+})$ and $\frac{d}{dt}Y^{\prime}_{r}(t)$ in $[\tau_{k},\tau_{k+1})$ from
\eqref{yrplus} and \eqref{dydt} respectively.\newline For $Z_{ij}(t)$ assuming
agent $j$ is the one connected to target $i$, we have:
\begin{equation}%
\begin{split}
&  Z^{\prime}_{ij}(\tau_{k}^{+})=Z^{\prime}_{ij}(\tau_{k}^{-})+\Big[\dot
Z_{ij}(\tau_{k}^{-})-\dot Z_{ij}(\tau_{k}^{+})\Big]\tau^{\prime}_{k}\\
&  =Z^{\prime}_{ij}(\tau_{k}^{-})+\Big[\tilde\mu_{ij}(\tau_{k}^{-})-\tilde
\mu_{ij}(\tau_{k}^{+})\Big]p_{ij}(\tau_{k})\tau^{\prime}_{k}=Z^{\prime}%
_{ij}(\tau_{k}^{-})
\end{split}
\label{zijplus}%
\end{equation}
In the above equation, $\tilde\mu_{ij}(\tau_{k}^{+})=\mu_{ij}$ because
$X_{i}(\tau_{k}^{+})>0$. Also, $\mu_{ij}p_{ij}(\tau_{k})=\sigma_{i}(\tau_{k})$
and $\tilde\mu_{ij}(\tau_{k}^{-})=\frac{\sigma_{i}(\tau_{k})}{p_{ij}(\tau
_{k})}$ results in $\tilde\mu_{ij}(\tau_{k}^{+})=\mu_{ij}$. For $Z_{il}(t)$,
$l\neq j$ , agent $l$ cannot be connected to target $i$ at $\tau_{k}$ so we
have, $Z^{\prime}_{il}(\tau_{k}^{+})=Z^{\prime}_{il}(\tau_{k}^{-})$ and
$\frac{d}{dt}Z^{\prime}_{il}(t)=0$ in $[\tau_{k},\tau_{k+1})$. For $Z_{rl}(t)$
,$r\neq i$ and $l\neq j$ using the assumption that two events occur at the
same time w.p. 0, the dynamics of $Z_{rl}(t)$ are not affected at $\tau_{k}$,
hence we get $\frac{d}{dt}Z^{\prime}_{rl}(t)$ from \eqref{dzdt} for $i$ and
$j$ replaced by $r$ and $l$.\newline3. \textbf{Event $\zeta_{ij}^{0}$}: This
event causes a transition from $Z_{ij}(t)>0$ for $t<\tau_{k}$ to $Z_{ij}(t)=0$
for $t \ge\tau_{k}$. The switching function is $g_{k}(\Theta,\mathbf{X}%
)=Z_{ij}(t)$ so $\frac{\partial g_{k}}{\partial Z_{ij}}=1$. From
\eqref{tauprime}:
\begin{equation}%
\begin{split}
&  {\tau_{k}}^{\prime}=-\Big(\frac{\partial g_{k}}{\partial Z_{ij}}f_{k}%
(\tau^{-}_{k})\Big)^{-1}\Big(g^{\prime}_{k}+\frac{\partial g_{k}}{\partial
Z_{ij}}{Z_{ij}^{\prime}(\tau^{-}_{k})}\Big)\\
&  =-\frac{Z_{ij}^{\prime}(\tau^{-}_{k})}{\tilde\mu_{ij}(\tau^{-}_{k}%
)p_{ij}(\tau^{-}_{k})-\beta_{ij}p_{\!_{Bj}}(\tau^{-}_{k})}=\frac
{Z_{ij}^{\prime}(\tau^{-}_{k})}{\beta_{ij}p_{\!_{Bj}}(\tau_{k})}%
\end{split}
\end{equation}
Since $Z_{ij}(t)$ is being emptied at $\tau_{k}$, by the assumption that
agents can not be in range with the base and targets at the same time, we have
$p_{ij}(\tau_{k})=0$. Then from \eqref{Xprime}:
\begin{equation}%
\begin{split}
&  Z^{\prime}_{ij}(\tau_{k}^{+})=Z^{\prime}_{ij}(\tau_{k}^{-})+\Big[-\beta
_{ij}p_{\!_{Bj}}(\tau_{k})-0\Big]{\tau_{k}}^{\prime}\\
&  =Z^{\prime}_{ij}(\tau_{k}^{-})-\Big[\beta_{ij}p_{\!_{Bj}}(\tau
_{k})\Big]\frac{Z_{ij}^{\prime}(\tau^{-}_{k})}{\beta_{ij}p_{\!_{Bj}}(\tau
_{k})}=0
\end{split}
\label{Zprime0}%
\end{equation}
Since $\dot Z_{ij}(t)=0$ in $[\tau_{k},\tau_{k+1})$:
\begin{equation}
\frac{d}{dt}Z^{\prime}_{ij}(t)=\frac{\partial\dot Z_{ij}(t)}{\partial
Z_{ij}(t)}Z^{\prime}_{ij}(t)+\dot Z^{\prime}_{ij}(t)=0 \label{Zprimet0}%
\end{equation}
For $Z_{rl}(t)$, $r \ne i$ or $l \ne j$, the dynamics in \eqref{Zdot} are not
affected at $\tau_{k}$, hence:
\begin{equation}
Z^{\prime}_{rl}(\tau_{k}^{+})=Z^{\prime}_{rl}(\tau_{k}^{-}) \label{Zrlplus}%
\end{equation}
if $Z_{rl}(\tau_{k})>0$, the value for $\frac{d}{dt}Z^{\prime}_{rl}(t)$ is
calculated by \eqref{dzdt} with $r$ and $l$ replacing $i$ and $j$
respectively. If $Z_{rl}(\tau_{k})=0$ then $\frac{d}{dt}Z^{\prime}_{rl}(t)=0$.
\newline For $Y_{i}(t)$ we have $\beta_{i}(\tau_{k}^{+})=0$ since the agent
has emptied its queue, hence:
\begin{equation}%
\begin{split}
&  Y^{\prime}_{i}(\tau_{k}^{+})=Y^{\prime}_{i}(\tau_{k}^{-})+\Big[\dot
Y_{i}(\tau_{k}^{-})-\dot Y_{i}(\tau_{k}^{+})\Big]\tau^{\prime}_{k}\\
&  =Y^{\prime}_{i}(\tau_{k}^{-})+[\beta_{ij}p_{\!_{Bj}}(\tau_{k}%
)-0]\frac{Z_{ij}^{\prime}(\tau^{-}_{k})}{\beta_{ij}p_{\!_{Bj}}(\tau_{k})}\\
&  =Y^{\prime}_{i}(\tau_{k}^{-})+Z_{ij}^{\prime}(\tau^{-}_{k})
\end{split}
\label{Yprime}%
\end{equation}
In $[\tau_{k},\tau_{k+1})$ we can get $\frac{d}{dt}Y^{\prime}_{i}(t)=0$. For
$Y_{r}(t)$, $r\neq i$ the dynamics of $Y_{r}(t)$ in \eqref{Ydot} are not
affected by the event at $\tau_{k}$ hence, $Y^{\prime}_{r}(\tau_{k}^{+})$ and
$\frac{d}{dt}Y^{\prime}_{r}(t)$ in $[\tau_{k},\tau_{k+1})$ are calculated from
\eqref{yrplus} and \eqref{dydt} respectively. The dynamics of $X_{r}(t)$,
$r=1,\dots,M$ is are not affected at $\tau_{k}$ since the event at $\tau_{k}$
is happening at the base. We have $X_{r}^{\prime}(\tau_{k}^{+})=X_{r}^{\prime
}(\tau_{k}^{-})$. If $X_{r}(\tau_{k})>0$ then we have $\frac{d}{dt}X^{\prime
}_{r}(t)$ from \eqref{dxdt} and if $X_{r}(\tau_{k})=0$ then $\frac{d}%
{dt}X^{\prime}_{r}(t)=0$ in $[\tau_{k},\tau_{k+1})$.\newline4. \textbf{Event
$\delta_{ij}^{+}$}: This event causes a transition from $D^{+}_{ij}(t)=0$ for
$t\le\tau_{k}$ to $D^{+}_{ij}(t)>0$ for to $t>\tau_{k}$. It is the moment that
agent $j$ leaves target $i$'s range. The switching function is $g_{k}(\Theta,
\mathbf{X})=d_{ij}(t)-r_{ij}$ , from \eqref{tauprime}:
\begin{equation}
{\tau_{k}}^{\prime}=- \frac{\partial d_{ij}}{\partial s_{j}}s_{j}^{\prime
}(t)\Big(\frac{\partial d_{ij}}{\partial s_{j}}\dot s_{j}(\tau_{k})\Big)^{-1}
\label{tauprime4}%
\end{equation}
If agent $j$ was connected to target $i$ at $\tau_{k}$ then by leaving the
target, it is possible that another agent $l$ which is within range with
target $i$ connects to that target. This means $\dot X_{i}(\tau_{k}%
^{+})=\sigma_{i}(\tau_{k})-\mu_{il}p_{il}(\tau_{k})$ and $\dot X_{i}(\tau
_{k}^{-})=\sigma_{i}(\tau_{k})-\mu_{ij}p_{ij}(\tau_{k})$, with $p_{ij}%
(\tau_{k})=0$, from \eqref{Xprime} we have
\begin{equation}
X_{i}^{\prime}(\tau_{k}^{+})= X_{i}^{\prime}(\tau_{k}^{-})-\mu_{il}p_{il}%
(\tau_{k})\tau_{k}^{\prime} \label{Xplusjump}%
\end{equation}
If $X_{i}(\tau_{k})>0$, $\frac{d}{dt}X^{\prime}_{i}(t)$ in $[\tau_{k}%
,\tau_{k+1})$ is as in \eqref{dxdt} with $r$ replaced by $i$ and if
$X_{i}(\tau_{k})=0$ then $\frac{d}{dt}X^{\prime}_{i}(t)=0$. On the other hand,
if agent $j$ was not connected to target $i$ at $\tau_{k}$, we know that some
$l\neq j$ is already connected to target $i$. This means agent $j$ leaving
target $i$ cannot affect the dynamics of $X_{i}(t)$ so we have $X_{i}^{\prime
}(\tau_{k}^{+})=X_{i}^{\prime}(\tau_{k}^{-})$ and $\frac{d}{dt}X_{i}^{\prime
}(t)$ is calculated from \eqref{dxdt} with $r$ replaced by $i$.\newline For
$X_{r}(t)$, $r \ne i$ the dynamics in \eqref{Xdot} are not affected by the
event at $\tau_{k}$ hence, we get $X_{r}^{\prime}(\tau_{k}^{+})$ from
\eqref{Xrprime0}. If $X_{r}(\tau_{k})>0$ the time derivative $\frac{d}%
{dt}X^{\prime}_{r}(t)$ in $[\tau_{k},\tau_{k+1})$ can be calculated from
\eqref{dxdt} and if $X_{r}(\tau_{k})=0$ then $\frac{d}{dt}X^{\prime}_{r}%
(t)=0$.\newline For $Y_{r}(t)$, $r=1,\ldots,,M$, the dynamics in \eqref{Ydot}
are not also affected by the event at $\tau_{k}$ hence, we get $Y_{r}(\tau
_{k}^{+})$ from \eqref{yrplus} and in $[\tau_{k},\tau_{k+1})$ the $\frac
{d}{dt}Y^{\prime}_{r}(t)$ is calculated from \eqref{dydt}.\newline For
$Z_{ij}(t)$, the dynamics in \eqref{Zdot} are not affect at $\tau_{k}$,
regardless of the fact that agent $j$ is connected to target $i$ or not. We
have $\dot Z_{ij}(\tau_{k}^{-})=\tilde\mu_{ij}(\tau_{k})p_{ij}(\tau_{k})$ with
$p_{ij}(\tau_{k})=0$ and $\dot Z_{ij}(\tau_{k}^{+})=0$, hence from
\eqref{Xprime}:
\begin{equation}%
\begin{split}
&  Z^{\prime}_{ij}(\tau_{k}^{+})=Z^{\prime}_{ij}(\tau_{k}^{-})+\Big[\dot
Z_{ij}(\tau_{k}^{-})-\dot Z_{ij}(\tau_{k}^{+})\Big]\tau^{\prime}_{k}\\
&  =Z^{\prime}_{ij}(\tau_{k}^{-})+\tilde\mu_{ij}(\tau_{k})p_{ij}(\tau_{k}%
)\tau^{\prime}_{k}=Z^{\prime}_{ij}(\tau_{k}^{-})
\end{split}
\end{equation}
and in $[\tau_{k},\tau_{k+1})$ , we have $\frac{d}{dt}Z^{\prime}_{ij}(t)=0$
using \eqref{dzdt} knowing $p_{ij}(\tau_{k})=p_{\!_{Bj}}(\tau_{k})=0$. For
$Z_{rl}(t)$, $r \neq i$ or $l \neq j$, the dynamics of $Z_{rl}(t)$ are not
affected at $\tau_{k}$ hence \eqref{Zrlplus} holds and in $[\tau_{k}%
,\tau_{k+1})$ again we can use \eqref{dzdt} with $i$ and $j$ replaced by $r$
and $l$.\newline5. \textbf{Event $\delta_{ij}^{0}$}: This event causes a
transition from $D^{+}_{ij}(t)>0$ for $t<\tau_{k}$ to $D^{+}_{ij}(t)=0$ for to
$t\ge\tau_{k}$. The event is the moment that agent $j$ enters target $i$'s
range. The switching function is $g_{k}(\Theta,\mathbf{X})=d_{ij}(t)-r_{ij}$.
From \eqref{tauprime} we can get ${\tau_{k}}^{\prime}$ from \eqref{tauprime4}.
If no other agent is already connected to target $i$, agent $j$ connects to
it. Otherwise, if another agent is already connected to target $i$, no
connection is established. For $X_{i}(t)$, the dynamics in \eqref{Xdot} are
not affected in both cases, hence, \eqref{xiplus} holds. If $X_{i}(t)>0$ in
$[\tau_{k},\tau_{k+1})$ we calculate $\frac{d}{dt}X^{\prime}_{i}(t)$ using
\eqref{dxdt} with $l$ being the appropriate connected agent to target $i$. If
$X_{i}(\tau_{k}^{-})=0$, $\frac{d}{dt}X^{\prime}_{i}(t)=0$. For $X_{r}(t)$, $r
\ne i$ the dynamics in \eqref{Xdot} are not affected by the event at $\tau
_{k}$. Hence, we get $X_{r}^{\prime}(\tau_{k}^{+})$ from \eqref{Xrprime0}. If
$X_{r}(\tau_{k})>0$ we calculate $\frac{d}{dt}X^{\prime}_{r}(t)$ from
\eqref{dxdt} with $i$ replaced by $r$ and if $X_{r}(\tau_{k})=0$ then
$\frac{d}{dt}X^{\prime}_{r}(t)=0$.\newline For $Y_{r}(t)$, $r=1,\dots,M$ again
the dynamics in \eqref{Ydot} are not affected at $\tau_{k}$ so both
\eqref{yrplus} and \eqref{dydt} hold.\newline For $Z_{ij}(t)$, with agent $j$
being connected or not to target $i$ at $\tau_{k}$ the dynamics of $Z_{ij}(t)$
are unaffected at $\tau_{k}$, hence \eqref{Zrlplus} holds for $i$ and $j$ and
in $[\tau_{k},\tau_{k+1})$ the $\frac{d}{dt}Z^{\prime}_{ij}(t)$ is calculated
through \eqref{dzdt}. For $Z_{rl}(t)$, $r \neq i$ or $l \neq j$ the dynamics
are unaffected \eqref{Zrlplus} holds again. In $[\tau_{k},\tau_{k+1})$,
$\frac{d}{dt}Z^{\prime}_{rl}(t)$ is given through \eqref{dzdt} with $i$ and
$j$ replaced by $r$ and $l$.\newline6. \textbf{Event $\Delta_{j}^{+}$}: This
event causes a transition from $D^{+}_{\!{Bj}}(t)=0$ for $t\le\tau_{k}$ to
$D^{+}_{\!{Bj}}(t)\ge0$ for $t> \tau_{k}$. The switching function is
$g_{k}(\Theta,\mathbf{X})=d_{\!_{Bj}}(t)-r_{\!_{Bj}}$.
\begin{equation}
{\tau_{k}}^{\prime}=-\frac{\partial D{\!_{Bj}}}{\partial s_{j}}s^{\prime}%
_{j}(\tau_{k})\Big(\frac{\partial d_{\!_{Bj}}}{\partial s_{j}}\dot s_{j}%
(\tau_{k})\Big)^{-1} \label{tauprime6}%
\end{equation}
Similar to the previous event, the dynamics of $X_{i}(t)$ are unaffected at
$\tau_{k}$ hence, we have $X_{i}^{\prime}(\tau_{k}^{+})$ calculated from
\eqref{xiplus}. If $X_{i}(t)>0$ in $[\tau_{k},\tau_{k+1})$ we calculate
$\frac{d}{dt}X^{\prime}_{i}(t)$ through \eqref{dxdt} and if $X_{i}(\tau
_{k}^{-})=0$, $\frac{d}{dt}X^{\prime}_{i}(t)=0$.\newline For $Y_{r}(t)$,
$r=1,\ldots,,M$, the dynamics of $Y_{r}(t)$ in \eqref{Ydot} are not affected
at $\tau_{k}$, hence, we get $Y_{r}(\tau_{k}^{+})$ from \eqref{yrplus} and in
$[\tau_{k},\tau_{k+1})$, $\frac{d}{dt}Y^{\prime}_{r}(t)$ is calculated from
\eqref{dydt}.\newline For $Z_{ij}(t)$, Using the fact that agent $j$ can only
be connected to one target or the base, we have $\dot Z_{ij}(\tau_{k}%
^{-})=\beta_{ij}(\tau_{k})p_{\!_{Bj}}(\tau_{k})$ with $p_{\!_{Bj}}(\tau
_{k})=0$ and $\dot Z_{ij}(\tau_{k}^{+})=0$, hence \eqref{Zrlplus} holds with
$i$ and $j$ replacing $r$ and $l$. In $[\tau_{k},\tau_{k+1})$ from
\eqref{dXdtprime}:
\begin{equation}%
\begin{split}
&  \frac{d}{dt}Z^{\prime}_{ij}(t)=\frac{\partial\dot Z_{ij}(t)}{\partial
Z_{ij}(t)}Z^{\prime}_{ij}(t)+\dot Z^{\prime}_{ij}(t)\\
&  =\dot Z^{\prime}_{ij}(t)=-\beta_{ij} P^{\prime}_{\!_{Bj}}(t)
\end{split}
\label{ZprimetB}%
\end{equation}
As for $Z_{rl}(t)$, $r \neq i$ or $l \neq j$ the dynamics are unaffected so
\eqref{Zrlplus} holds. In $[\tau_{k},\tau_{k+1})$ we can calculate $\frac
{d}{dt}Z^{\prime}_{rl}(t)$ through \eqref{dzdt} with $j$ replacing
$l$.\newline7. \textbf{Event $\Delta_{j}^{0}$}: This event causes a transition
from $D^{+}_{\!{Bj}}(t)>0$ for $t<\tau_{k}$ to $D^{+}_{\!{Bj}}(t)=0$ for
$t\ge\tau_{k}$. The switching function is $g_{k}(\Theta, \mathbf{X}%
)=d_{\!_{Bj}}(t)-r_{\!_{Bj}}$. Using \eqref{tauprime} we can get ${\tau_{k}%
}^{\prime}$ from \eqref{tauprime6}. Similar with the previous event we have
$X_{i}^{\prime}(\tau_{k}^{+})$ from \eqref{xiplus}. If $X_{i}(t)>0$ we can get
$\frac{d}{dt}X^{\prime}_{i}(t)$ from \eqref{dxdt} and if $X_{i}(\tau_{k}%
^{-})=0$ then $\frac{d}{dt}X^{\prime}_{i}(t)=0$.\newline For $Y_{r}(t)$,
$r=1,\ldots,,M$, we again follow the previous event analysis so \eqref{yrplus}
and \eqref{dydt} hold.\newline For $Z_{ij}(t)$, the analysis is similar to
event $\Delta_{j}^{+}$ so we can calculate $Z_{ij}^{\prime}(\tau_{k}^{+})$ and
$\frac{d}{dt}Z^{\prime}_{ij}(t)$ in $[\tau_{k},\tau_{k+1})$ from
\eqref{zijplus} and \eqref{dzdt} respectively. Also for $Z_{rl}(t)$, $r \neq
i$ or $l \neq j$, \eqref{Zrlplus} holds with same reasoning as previous event.
In $[\tau_{k},\tau_{k+1})$ we calculate $\frac{d}{dt}Z^{\prime}_{rl}(t)$ from \eqref{dzdt}.

\section{Objective function gradient}

\label{Jgradient} From \eqref{gradJ} we have:
\begin{equation}%
\begin{split}
&  \nabla\mathcal{L}(\Theta,T; \mathbf{X}(\Theta; 0)))=\frac{1}{T}%
\Big[\sum\limits_{k=0}^{K}\int_{\tau_{k}}^{\tau_{k+1}}\Big(q\nabla
\mathcal{L}_{1}(\Theta,t)\\
&  -(1-q)\nabla\mathcal{L}_{2}(\Theta,t)+\nabla\mathcal{L}_{3}(\Theta
,t)+\nabla\mathcal{L}_{4}(\Theta,t)\Big)dt\Big]\\
&  +\nabla\mathcal{L}_{f}(\Theta,T)
\end{split}
\end{equation}
We calculate each term separately:
\begin{equation}%
\begin{split}
\nabla\mathcal{L}_{1}(\Theta,t)=\frac{1}{M_{X}}\sum\limits_{i=1}^{M}\alpha
_{i}X^{\prime}_{i}(t)
\end{split}
\end{equation}
\begin{equation}%
\begin{split}
\nabla\mathcal{L}_{2}(\Theta,t)=\frac{1}{M_{Y}}\sum\limits_{i=1}^{M}\alpha
_{i}Y^{\prime}_{i}(t)
\end{split}
\end{equation}
\begin{equation}
\resizebox{0.99 \columnwidth}{!}{$
\begin{split}
\nabla &  \mathcal{L}_{3}(\Theta,t)=\\
&  \frac{1}{M_{I}I_{j}(t)}\Bigg({d_{\!_{Bj}}^{+}}^{\prime}(t)\prod_{i=1}%
^{M}d_{ij}^{+}(t)+d_{\!_{Bj}}^{+}(t)\sum_{l=1}^{M} {d_{lj}^{+}}^{\prime
}(t)\prod_{i=1,i\ne l}^{M}d_{ij}^{+}(t)\Bigg)
\end{split}$}
\end{equation}
\begin{equation}
\resizebox{0.99 \columnwidth}{!}{$
\begin{split}
\nabla\mathcal{L}_{4}  &  (\Theta,t)=\frac{1}{M_{R}}\sum_{j=1}^{N}
\Big[\int_{S} \bigg(R(w,t)+ R_{\!_{Bj}}(w,t)\bigg)P_{j}^{\prime}(w,t)dw\\
&  \qquad+\int_{S}\bigg(R^{\prime}(w,t)+ R_{\!_{Bj}}^{\prime}(w,t)\bigg)P_{j}%
(w,t)dw\Big]\\
&  =\frac{1}{M_{R}}\sum_{j=1}^{N} \Big[\int_{S}\bigg(R(w,t)+ R_{\!_{Bj}%
}(w,t)\bigg)2\langle s_{j}(t)-w, s_{j}^{\prime}(t)\rangle dw\\
&  \qquad+\int_{S}\bigg(\sum_{i=1}^{M}\frac{\alpha_{i} X_{i}^{\prime}%
(t)}{d_{i}^{+}(w)}+\frac{\sum_{i=1}^{M} \alpha_{i} Z_{ij}^{\prime}}{d_{\!_{B}%
}^{+}(w)}\bigg)P_{j}(w,t)dw\Big]
\end{split}$}
\end{equation}
\begin{equation}%
\begin{split}
\nabla\mathcal{L}_{f}(\Theta,T)=\frac{1}{M_{Z}}\sum\limits_{i=1}^{M}\alpha
_{i}Z^{\prime}_{ij}(T)
\end{split}
\end{equation}
\end{appendices}

\bibliographystyle{hieeetr}
\bibliography{Ref_DH}

\end{document}